\documentclass[11pt]{amsart}
\usepackage[margin=1in]{geometry}

\usepackage{amssymb,amscd,amsmath,latexsym}
\usepackage[mathcal]{euscript}
\usepackage{hyperref}
\usepackage{xcolor}
\usepackage{ytableau}
\usepackage{tikz}
\usepackage{mfabacus}
\usepackage{mathtools}
\usepackage{bm}
\usepackage{orcidlink}

\usepackage[normalem]{ulem}

\newtheorem{theorem}{Theorem}[section]
\newtheorem{corollary}[theorem]{Corollary}
\newtheorem{lemma}[theorem]{Lemma}
\newtheorem{claim}[theorem]{Claim}
\newtheorem{proposition}[theorem]{Proposition}
\newtheorem{example}[theorem]{Example}
\theoremstyle{definition}
\newtheorem{definition}[theorem]{Definition}
\newtheorem{conjecture}[theorem]{Conjecture}

\numberwithin{equation}{section}
\numberwithin{figure}{section}

\makeatletter
\@namedef{subjclassname@2020}{%
	\textup{2020} Mathematics Subject Classification}
\makeatother

\newcommand{\Ind}{\big\uparrow}
\newcommand{\Res}{\big\downarrow}
\newcommand{\sgn}{\operatorname{sgn}}
\DeclareMathOperator{\emp}{emp}
\DeclareMathOperator{\Add}{Add}
\DeclareMathOperator{\Aone}{A1}
\DeclareMathOperator{\Atwo}{A2}

\theoremstyle{remark}
\newtheorem{remark}[theorem]{Remark}

\title{New columns in decomposition matrices of symmetric groups for every block}

\author[David J. Hemmer]{David J. Hemmer \orcidlink{0000-0002-7411-4495}}
\address{Department of Mathematical Sciences\\
  Michigan Technological University\\
  Houghton, MI 49931}
\email{djhemmer@mtu.edu}

\author[Pavel Turek]{Pavel Turek \orcidlink{0000-0002-6190-0745}}
\address{Representation Theory and Algebraic Combinatorics Unit\\
  Okinawa Institute of
Science and Technology\\
  Onna, Okinawa, Japan 904-0495}
\email{pavel.turek@oist.jp}


\begin{document}

\begin{abstract} The central unsolved problem in the modular representation theory of symmetric groups is to find the decomposition matrices, which describe how irreducible representations in characteristic zero decompose upon reduction modulo a prime characteristic $p$. In this paper we determine a large number of new columns in these decomposition matrices, namely those labeled by partitions whose $p$-divisible hooks have all even arm lengths. In particular in odd characteristic $p$, for every possible block of every possible symmetric group $S_n$, we determine at least one complete column. These columns are multiplicity-free and are described by a recently introduced combinatorial statistic of partitions (depending on $p$), called the \textit{odd sequence}. As an application, we determine the indecomposable summands of Foulkes modules $H^{(2^m)}$.

\end{abstract}

\maketitle

\section{Introduction and Main Results}\label{sec:intro} Let $S_n$ denote the symmetric group on $n$ letters. Our notation follows James' classic book \cite{JamesSymmetric78} on the representation theory of symmetric groups. For any partition $\lambda$ of $n$ let $S^\lambda$ denote the corresponding Specht module. These modules are irreducible in characteristic zero, but defined over $\mathbb{Z}$, and hence over any field. The irreducible modules in characteristic $p$ are labeled by $p$-regular partitions $\mu$ and are denoted $D^\mu$. The decomposition number $d_{\lambda\mu}$ records the multiplicity of $D^\mu$ in $S^\lambda$. There has been an enormous amount of work done trying to understand these numbers, commonly organized into matrices called decomposition matrices; see Section~\ref{sec:background} or \cite[Section 1.1]{GiannelliWildonFoulkesandDecomposition15} for a summary.

Let $p$ be an odd prime. Our first main result describes new columns of decomposition matrices in characteristic $p$. It applies to a large number of partitions $\mu$, namely to $\binom{w+(p-3)/2}{w}$ partitions in any fixed block of $p$-weight $w$ (see \cite[Theorem~5.1]{HemmerTurekArms26}); in particular, at least one in every block of every symmetric group! The authors believe that this is the first explicit description of whole columns of decomposition matrices in each block. Moreover, it may be the most elementary formula for decomposition numbers to date.

To state it, we define the \textit{odd sequence} of a partition $\lambda$ to be a sequence $(n_0, n_1, \dots, n_{p-1})$ where $n_i$ is the number of odd parts $\lambda_j$ of $\lambda$ with $\lambda_j - j \equiv i$ (mod $p$). We also refer to hooks of a partition of size divisible by $p$ as \textit{$p$-divisible hooks}. Both decomposition numbers and odd sequences depend on $p$; however, we omit this dependence from the notation, as $p$ is always implicit from the context.

\begin{theorem}\label{theorem:even arms}
    Let $p$ be an odd prime and $\lambda$ and $\mu$ be two partitions of the same size, with $\mu$ $p$-regular. If all arm lengths of the $p$-divisible hooks of $\mu$ are even, then
    \begin{align*}
        d_{\lambda\mu} = \begin{cases}
            1 & \text{if $\lambda$ and $\mu$ have equal $p$-cores and odd sequences;} \\
            0 & \text{otherwise.}
        \end{cases}
    \end{align*}
\end{theorem}

\begin{remark}\label{re:p=2}
    As stated, Theorem~\ref{theorem:even arms} is true even if $p=2$, but is rather boring as it only applies to $2$-core partitions $\mu$.
\end{remark}

Thus, once $\mu$ satisfies the even-arm condition, the entire column labeled by $\mu$ is controlled only by two elementary combinatorial statistics: the $p$-core and the odd sequence.

While our focus is on symmetric groups, Theorem~\ref{theorem:even arms} also holds for decomposition numbers of Hecke algebras of quantum characteristic $p$ over a field of characteristic $0$; see Corollary~\ref{cor:Hecke}. This easily follows from some of our arguments and the properties of adjustment matrices.

Theorem~\ref{theorem:even arms} can be restated using \cite[Proposition~1.6]{TurekMullineuxArxiv25}, which gives an equivalent description of partitions with all arm lengths of its $p$-divisible hooks even (referred to as $2$-balanced and $2$-shift skewed partitions in \cite{TurekMullineuxArxiv25}). We refer the reader to Proposition~\ref{pr:equivalence max} and its preceding and subsequent comments for a detailed explanation. Before presenting the restated version, we motivate and connect it to the work of Giannelli and Wildon \cite{GiannelliWildonFoulkesandDecomposition15}.

In that paper the authors gave the following definitions. Let $p$ be an odd prime, $\gamma$ a $p$-core, and $k \geq 0$ and integer. Let $w_k(\gamma)$ be the minimum number of rim $p$-hooks that may be added to $\gamma$ to obtain a partition with exactly $k$ odd parts. Let $\mathcal{E}_k(\gamma)$ denote all partitions with $p$-core $\gamma$, $p$-weight $w_k(\gamma)$ and exactly $k$ odd parts. They proved:

\begin{theorem}\cite[Theorem 1.1]{GiannelliWildonFoulkesandDecomposition15}
\label{theorem:GWThm1.1} Let $p, \gamma, k$ be as above. If $k \geq p$ suppose further that $w_{k-p}(\gamma) \neq w_k(\gamma)-1.$ Then $\mathcal{E}_k(\gamma)$ is a disjoint union of subsets $\mathcal{X}_1,\mathcal{X}_2, \ldots, \mathcal{X}_c$ so that each $\mathcal{X}_i$ has a unique maximal partition $\nu_i$ in the dominance order. Each $\nu_i$ is $p$-regular and the column of the symmetric group decomposition matrix in characteristic $p$ labeled by $\nu_i$ has ones in the rows labeled by partitions in $\mathcal{X}_i$ and zeros in all other rows.
\end{theorem}

\begin{remark}
\label{remark:GianelliWildononlygivesunionofcolumns}
Theorem \ref{theorem:GWThm1.1} does not in general give individual columns of the decomposition matrix, but instead gives the entries in a sum of some columns. In their Proposition 6.4, they prove that for each integer $w\geq 0$ there exists some block of $p$-weight $w$ where $c=1$, i.e. their result gives an actual column.
\end{remark}
We refine the number of odd parts statistic by tracking also the $p$-residues of $\lambda_j-j$ for each odd part $\lambda_j$, through the odd sequence and a corresponding set of partitions $\mathcal{E}_\theta(\gamma)$ (defined properly in Section \ref{sec:abacus}). It turns out this statistic completely describes the sets $\mathcal{X}_i$ in Theorem \ref{theorem:GWThm1.1}, resolving its inexplicitness mentioned in Remark \ref{remark:GianelliWildononlygivesunionofcolumns} (see Lemma~\ref{le:GWCorollary} and Remark~\ref{remark:GWSets}). It also gives an enormous number of other new columns in decomposition matrices. This is what Theorem~\ref{theorem:even arms} does, though its connection to the work of Giannelli--Wildon is better spotted from its promised reformulation.

\begin{theorem}\label{theorem:maintheorem}
  Let $p$ be an odd prime. Let $\theta=(\theta_0,\theta_1,\ldots, \theta_{p-1})$ be a composition such that at least one $\theta_i=0$, and let $\gamma$ be a $p$-core partition. There is a unique maximal element in $\mathcal{E}_\theta(\gamma)$ in the dominance order, which is $p$-regular. The corresponding column of the symmetric group decomposition matrix in characteristic $p$ has a one in precisely the rows labeled by partitions in $\mathcal{E}_\theta(\gamma)$ and zeros elsewhere.
\end{theorem}

Several examples of columns identified by Theorem~\ref{theorem:maintheorem} (or, equivalently, Theorem~\ref{theorem:even arms}) are in Section~\ref{subsection:Examples}.

We now fix a field $F$. The \textit{Foulkes module} $H^{(2^m)}$ is the permutation $FS_{2m}$-module on the cosets of the wreath product $S_2\wr S_m\leq S_{2m}$; that is $H^{(2^m)} \cong F\Ind^{S_{2m}}$, where $F$ is considered as the trivial $F(S_2\wr S_m)$-module for the induction. The Foulkes modules and their decompositions are commonly studied due to their connection to the long-standing plethysm problem; see the end of Section~\ref{sec:background}. Here, our final main result describes the indecomposable summands of $H^{(2^m)}$ when $F$ has characteristic $p$. Put concisely, it states that $H^{(2^m)}$ has exactly one indecomposable summand in each possible block. In the statement (and throughout the paper), we write $V_{\tau}$ for the indecomposable summand of $FS_n$-module $V$ lying in the block labeled by $\tau$, a $p$-content of a partition of $n$ (see Theorem~\ref{thm:NakayamaF} and the subsequent comments) and refer to partitions with no odd parts as \textit{even partitions}.  

\begin{theorem}\label{thm:indecomposables}
    Let $p$ be an odd prime and $m\geq 0$ an integer. Let $\tau$ be the $p$-content of an even partition of $2m$. Then over a field of characteristic $p$, $H^{(2^m)}_{\tau}$ is indecomposable. That is,
    \[
    H^{(2^m)} = \bigoplus_{\tau} H^{(2^m)}_{\tau},
    \]
    is the decomposition of the Foulkes module $H^{(2^m)}$ into indecomposable summands (where the sum runs over all distinct $p$-contents $\tau$ of even partitions of $2m$).
\end{theorem}

\subsection{Outline of Proof}
\label{sec:outlineproof}
Key objects studied in \cite{GiannelliWildonFoulkesandDecomposition15} and used here are the twisted Foulkes modules: given integers $m,k\geq 0$, the \textit{twisted Foulkes module} $H^{(2^m;k)}$ is defined as

$$H^{(2^m;k)}=(H^{(2^m)} \boxtimes \sgn_{S_k})\Ind^{S_{2m+k}}_{S_{2m}\times S_k}.$$
It is known that the ordinary character of $H^{(2^m;k)}$ decomposes as the sum of the ordinary characters of Specht modules labeled by $\lambda \vdash 2m+k$ with precisely $k$ odd parts. Through a careful study of the vertices and Green correspondents of the summands of these modules (in characteristic $p$), Giannelli and Wildon are able to find a projective summand in an appropriate block. Using the crucial property that the multiplicity of a Specht module $S^{\lambda}$ in the projective cover $P^{\mu}$ of $D^{\mu}$ coincides with the decomposition number $d_{\lambda\mu}$ (see \eqref{eq:projective} for exact formulation), they immediately deduce Theorem \ref{theorem:GWThm1.1}. They cannot however determine what the indecomposable summands of this projective summand are, which would produce the partition into the sets $\mathcal{X}_1, \mathcal{X}_2,\dots, \mathcal{X}_c$.

Using two block projections instead of one, we consider a finer direct summand of a twisted Foulkes module with an ordinary character equal to the sum of ordinary characters of $S^{\lambda}$ labeled by partitions $\lambda\in\mathcal{E}_\theta(\gamma)$. We are able to use the combinatorics of the odd sequences, studied in detail in \cite{TurekMullineuxArxiv25}, together with the Jantzen--Schaper formula to ensure that, if the summand we find is projective, then it must be indecomposable. This argument is fully combinatorial and uses a sign-reversing involution. Then a careful study of vertices, using the Brauer morphism, allows us to conclude this summand is indeed projective, and then conclude Theorem~\ref{theorem:maintheorem} (and, its equivalent statement, Theorem~\ref{theorem:even arms}); in fact, not only do we determine the decomposition matrix column labeled by $\mu$ in Theorem~\ref{theorem:even arms}, we also conclude that our direct summand of $H^{(2^m;k)}$ is the projective module $P^{\mu}$, providing an explicit construction of $P^{\mu}$.

This final step is analogous to the arguments used by Giannelli and Wildon, although due to the nature of our finer summands, we work in a more abstract setting. In particular, we prove and use the following `Mackey's theorem for the Brauer morphism', which, despite its simplicity and applicability, the authors were not able to find in the literature. We expect this formula to help find more projective modules, by following the strategy we take in Section~\ref{sec:Brauer}.

\begin{proposition}\label{pr:Mackey for Brauer}
    Let $F$ be a field of prime characteristic $p$, $K\leq G$ be two finite groups, and $R\leq G$ be a $p$-group. Then for a $p$-permutation $FK$-module $V$, we have
    \[
    V\Ind^G(R) \cong \bigoplus_{\substack{x\in N_G(R)\backslash G/K \\ R\leq \prescript{x}{}{K}}} \left(\prescript{x}{}{V}\right)(R)\Ind^{N_G(R)}.
    \]
\end{proposition}

Restricting attention to Foulkes modules $H^{(2^m)}$ and moving to Theorem~\ref{thm:indecomposables}, we use a formula due to Littlewood to show that the projective indecomposable summands we found are all the projective indecomposable summands of $H^{(2^m)}$. Using the description of the vertices of all the indecomposable summands of $H^{(2^m)}$ established by Giannelli--Wildon \cite[Theorem~1.2]{GiannelliWildonIndecomposable16}, we then easily deduce Theorem~\ref{thm:indecomposables}, describing all the indecomposable summands explicitly. In Conjecture~\ref{con:indecomposables}, we propose a generalization of this result to all twisted Foulkes modules $H^{(2^m; k)}$, namely, we conjecture that the indecomposable summands of $H^{(2^m; k)}$ are indexed by pairs of blocks, or equivalently by a pair of an $e$-core and an odd sequence. We expect the proof of Theorem~\ref{thm:indecomposables} to generalize; however, this may require a new combinatorial idea while using Littlewood's formula.

\subsection{Literature background}
\label{sec:background}
Here, we collect some known results about decomposition numbers, balanced partitions and Foulkes modules -- three topics our results concern.

\subsubsection*{Decomposition numbers}
As mentioned, \cite{GiannelliWildonFoulkesandDecomposition15} provides a well-written summary of the problem of computing decomposition numbers of symmetric groups. Many of the tools that can be used to compute some decomposition numbers are summarized in Mathas \cite[Section~6.4]{MathasHecke99} and Fayers \cite[Section~1]{FayersWeightThree08}. (In the former, their Specht modules $S^{\lambda}$ are the dual of Specht modules $S^{\lambda'}$, and their irreducible modules $D^{\mu}$ are our modules $D^{\mu'}$.) These tools apply more generally to decomposition numbers of Hecke algebras. Here we focus mainly on summarizing concrete known formulas for decomposition numbers, rather than on the various algorithms and tools that can be used to compute some unspecified decomposition numbers, in principle.

The decomposition numbers in blocks of $p$-weight $0$ and $1$ are well-understood (see, for example, \cite[Section~2]{FayersWeightThree08}). For $p$-weight $2$, when $p$ is odd, a combinatorial formula was found by Richards \cite[Theorem~4.4]{RichardsDecomposition96}. While Richards' formula is explicit, there is already an increase in complexity from the $p$-weight $1$ case, as it requires the computation of the Mullineux map. For $p=2$, the formulas for $p$-weight $2$ were found by Fayers \cite{FayersWeightTwoPTwo05}. Moving to $p$-weight $3$, there is no explicit combinatorial formula anymore; however, if $p\geq 5$, all the decomposition numbers are at most one, as proved by Fayers \cite[Theorem~1.1]{FayersWeightThree08}, and thus can be computed in principle using the Jantzen--Schaper formula in the multiplicity-free setting (see Section~\ref{section: JantzenSchaper}).

If instead $p=2$ or $p=3$, one can use the description of the adjustment matrices by Fayers and Tan \cite[Theorem~3.3]{FayersTanWeightThree06} to translate the problem to a computation of decomposition numbers of Hecke algebras over a field of characteristic $0$ and quantum characteristic $p$. Unlike in the symmetric groups setting, the decomposition numbers of Hecke algebras in characteristic $0$ can be computed (recursively) using the LLT algorithm by Lascoux, Leclerc and Thibon \cite{LascouxLeclercThibonCrystal96} and Ariki's theorem \cite[Theorem~4.4]{ArikiDecomposition96}; see \cite[Section~6.1]{MathasHecke99}. The same approach can be used for $p$-weight $4$ for $p\geq 5$; the adjustment matrix is the identity matrix as proved by Fayers \cite[Theorem~2.6]{FayersJamesWeightFour07}. This approach is far from explicit, and what is worse, cannot be generalized to an arbitrary $p$-weight: James' conjecture, which predicts that the adjustment matrix is the identity matrix whenever $p$ is greater than the $p$-weight, is false. A counterexample was first found by Williamson \cite{WilliamsonSchubert17}. A recent (smaller) counterexample was found by Speyer \cite{SpeyerCounterexampleJamesArxiv26}. These counterexamples keep the problem of computing all decomposition numbers of symmetric groups wide open. While talking about Hecke algebras in characteristic $0$, a family of explicit multiplicity-free columns of decomposition matrices of these Hecke algebras (labeled by so-called $4$-increasing partitions) was described by Chuang, Miyachi and Tan \cite{ChuangMiyachiTanTilings17}. Unlike our results, their columns do not appear in every block, but they still describe many columns, in particular, for blocks of weight much smaller than the underlying quantum characteristic.

Aside from blocks of small $p$-weight, there is a formula for the decomposition numbers of symmetric groups in the RoCK blocks (also known as the Rouquier block) of $p$-weight less than $p$. It was found by Chuang and Tan \cite[Theorem~6.2]{ChuangTanFiltration03}, using their previous study of wreath products \cite{ChuangTanWreath03} and the Morita equivalence between these RoCK blocks and principal blocks of wreath products proved by Chuang and Kessar \cite{ChuangKessarWreath02}. A more detailed summary about decomposition numbers of RoCK blocks (applicable to Hecke algebras) can be found in James, Lyle and Mathas \cite{JamesLyleMathasRoCK06}. Their Corollary~4.5 provides a formula for particular decomposition numbers even when the $p$-weight is more than or equal to $p$. We note that while the formula by Chuang and Tan \cite[Theorem~6.2]{ChuangTanFiltration03} is an explicit sum, it requires a computation of many Littlewood--Richardson coefficients.

There are some special columns of the decomposition numbers known in the literature. For example: columns labeled by hook partitions for $p$ odd found by Peel \cite[Theorem~2]{PeelHook71}, partitions of length $2$ found by James \cite{JamesDecompositionI76} and partitions of length $3$ with final part less than $p$ found by Williams \cite{WilliamsThreePart06}. In all these cases, the columns lie in very restricted families of blocks.

Given this summary of (more or less) explicitly known decomposition numbers, the authors believe that our Theorem~\ref{theorem:even arms} is the first to explicitly describe decomposition numbers in each block of symmetric groups (for odd $p$). Moreover, its surprising simplicity seems to only be matched by the formulas for decomposition numbers in blocks of $p$-weight $0$ and $1$ and in columns labeled by hook and length $2$ partitions. As such, Theorem~\ref{theorem:even arms} (and its equivalent version, Theorem~\ref{theorem:maintheorem}) provides useful data for studying decomposition numbers, and sheds more light on this long-standing question.

\subsubsection*{Balanced partitions}
The odd sequence and associated sets $\mathcal{E}_{\theta}(\gamma)$ were introduced in greater generality as the $d$-runner matrix $\mathcal{R}$ and associated sets $\mathcal{E}_{\mathcal{R}}(\gamma)$ by the second author \cite{TurekMullineuxArxiv25}. The $d$-runner matrix keeps track of the number of parts $\lambda_j$ such that $\lambda_j$ has a specified non-zero remainder modulo $d$ and $\lambda_j-j$ has a specified remainder modulo $p$ (in \cite{TurekMullineuxArxiv25} $p$ is replaced by any integer $e>1$, but for simplicity, we work with $p$ in this summary). The odd sequence then coincides with the $2$-runner matrix; we have chosen to use a new term here as `$2$-runner matrix' is, on its own, unnecessarily complicated and somewhat misleading, as the matrix has only one row.

The paper then studies the maximal and minimal elements of $\mathcal{E}_{\mathcal{R}}(\gamma)$ (with respect to the dominance order). In particular, it is shown that each $\mathcal{E}_{\mathcal{R}}(\gamma)$ has a unique maximal element, and the maximal elements of these sets coincide with \textit{$d$-shift skewed partitions} -- partitions such that the arm lengths of their $p$-divisible hooks are not congruent to $-1$ modulo $d$. Note that the condition in Theorem~\ref{theorem:even arms} imposed on $\mu$ is: $\mu$ is a $2$-shift skewed partition.

The second main result of \cite[Theorem~1.5]{TurekMullineuxArxiv25} then shows that if $d$-runner matrix $\mathcal{R}$ contains zero in each row, then, the Mullineux map $m_p$ maps the maximal element $\mu$ of $\mathcal{E}_{\mathcal{R}}(\gamma)$ to the conjugate of its minimal element, provided that $\mu$ is \textit{$d$-balanced} -- the arm lengths of its $p$-divisible hooks are divisible by $d$. For $d=2$, the final condition is unnecessary as $2$-balanced and $2$-shift skewed partitions are the same objects, and \cite[Theorem~1.5]{TurekMullineuxArxiv25} is an immediate consequence of our Theorem~\ref{theorem:maintheorem} due to the following fact: the bottommost row of the decomposition matrix (ordered by any linear extension of the dominance order) which contains a non-zero entry in column labeled by $\mu$ is labeled by the conjugate of $m_p(\mu)$.

In a parallel paper by the authors \cite{HemmerTurekArms26}, we find the number of $d$-balanced $p$-regular partitions in a given block of symmetric groups. Astonishingly, this number depends only on $d, p$ and the $p$-weight $w$ of the block (not on the $p$-core) -- it equals $\binom{w + \left\lfloor(p-1)/d\right\rfloor - 1}{w}$ -- and the result applies to any integers $d,p>1$. The proof of this quantity provides an algorithm for finding all the $\binom{w + \left\lfloor(p-1)/d\right\rfloor - 1}{w}$ $d$-balanced $p$-regular partitions in a given block; in particular, we can easily find $\binom{w + (p-3)/2}{w}$ columns in any given block of $p$-weight $w$ of the decomposition matrices in odd prime characteristic $p$ using our Theorem~\ref{theorem:even arms} and the implicit algorithm from \cite{HemmerTurekArms26}.

While some of our results here make sense and hold for any integer $d>1$ in place of $2$, there are several obstacles why we cannot extend our arguments to this more general setting. Firstly, the algorithms A1 and A2 from \cite{TurekMullineuxArxiv25}, used in Section~\ref{section:Algorithms}, are more complex for $d>2$ and do not fit the setting of the Jantzen--Schaper formula; however, this issue may be overcome and is a subject of an ongoing collaboration of Gustavsson, Law, Putignano, Speyer and the second author. Secondly, and more importantly, an algebraic interpretation for the case of $d>2$ is yet to be found (while for $d=2$, we can describe odd sequences using twisted Foulkes modules; see Lemma~\ref{le:FoulkesModules}). This, together with the use of the Brauer morphism, is also the reason why we restrict our attention to symmetric groups, rather than work with Hecke algebras. Adapting the results to Hecke algebras is also part of the ongoing work of Gustavsson, Law, Putignano, Speyer and the second author.

\subsubsection*{Foulkes modules}
As mentioned throughout Section~\ref{sec:intro}, our results are inspired by the work of Giannelli and Wildon \cite{GiannelliWildonFoulkesandDecomposition15}. In particular, we follow their overall strategy to compute new decomposition numbers using the Brauer morphism.  We, however, do not repeat the procedure of finding the vertices of the indecomposable summands of the twisted Foulkes modules, as they have already been described in \cite[Theorem~1.2]{GiannelliWildonFoulkesandDecomposition15}. Instead, we fully focus on finding their (projective) indecomposable summands, which gives our main theorems.

It is well known that the Foulkes modules $H^{(2^m)}$ have a Specht filtration given by Specht modules labeled by even partitions, when working over a field of characteristic $0$. In fact, the result holds in any characteristic; the explicit filtration is found by Paget \cite[Theorem~2]{PagetFiltration07}. In a more recent paper \cite{GiannelliWildonIndecomposable16}, Giannelli and Wildon prove more structural results about the Foulkes modules $H^{(2^m)}$ in odd characteristic $p$. In particular, they describe the indecomposable summands of $H^{(2^m)}$ in $p$-weight at most $2$ and when $2m<3p$, they also describe the Loewy series of the unique indecomposable summand in the principal block. Some of their results are now direct consequences of our Theorem~\ref{thm:indecomposables}.

Results from \cite{GiannelliWildonFoulkesandDecomposition15} were adapted to the setting of Foulkes modules $H^{(a^b)}$, defined as the permutation $FS_{ab}$-module on the cosets of $S_a\wr S_b$ in $S_{ab}$, by Giannelli \cite{GiannelliDecomposition15}. While the results about vertices and Green correspondents nicely translate to Foulkes modules $H^{(a^b)}$, the translated version of \cite[Theorem~1.1]{GiannelliWildonFoulkesandDecomposition15} is more abstract, because the ordinary character of $H^{(a^b)}$ is unknown in general; finding this character is equivalent to decomposing the plethysm $s_b\circ s_a$ which is a major open problem in algebraic combinatorics as identified by Stanley \cite[Problem~9]{StanleyPositivity00} and is the subject of the longstanding Foulkes' conjecture \cite{FoulkesConjPaper}.

\subsection{Outline of the paper}

We recall the notation and basic facts about partitions and representation theory of symmetric groups in Section~\ref{sec:abacus}. Anyone familiar with these topics should be able to skim through the section and use it mainly as a reference. Section~\ref{sec:odd} examines odd sequences and the sets $\mathcal{E}_{\theta}(\gamma)$; in particular, it shows their relations to summands of twisted Foulkes modules $H^{(2^m; k)}$. It also contains several examples of columns from Theorem~\ref{theorem:even arms} (or equivalently, Theorem~\ref{theorem:maintheorem}). We then recall algorithms A1 and A2 from \cite{TurekMullineuxArxiv25} in Section~\ref{section:Algorithms} and make them formally into inverse functions. These functions and the Jantzen--Schaper formula are then used in Section~\ref{section: JantzenSchaper}, when Theorem~\ref{theorem:maintheorem} is reduced to a statement that the columns in question are multiplicity-free.

Section~\ref{sec:Brauer} changes the combinatorial narrative to an algebraic one: we show that the summands of twisted Foulkes modules which correspond to sets $\mathcal{E}_{\theta}(\gamma)$ are projective. This is done using the Brauer morphism, which is the main player of the section. We note that Section~\ref{sec:Brauer} can be read independently of the previous two sections. Finally, we deduce our main results, some corollaries, and state the follow-up conjecture in Section~\ref{sec:proofs}.

\section{Preliminaries}
\label{sec:abacus}

\subsection{Partitions}
\label{subsection:Partitions}
In this subsection, $p$ is any positive integer (in the next subsection, we specialize to $p$ a prime, but this restriction is unnecessary for the partition combinatorics). We use standard composition and partition notation. In particular, for a composition or partition $\theta = (\theta_1,\theta_2,\dots,\theta_t)$, we refer to its entries as \textit{parts}, the number of its entries as \textit{length}, and the sum of its entries as \textit{size}. We write $|\theta|$ for its size, and say that $\theta$ is a composition (or partition) of $n$ if its size is $n$. This is also denoted by $\theta\vDash n$ for compositions and $\theta\vdash n$ for partitions. In examples, we usually group the same parts together, and use the notation $a^b$ for $b$ parts equal to $a$. Finally, we say that a composition $\theta$ \textit{contains a zero} if some $\theta_i=0$ (so $\theta$ in Theorem~\ref{theorem:maintheorem} is assumed to be a composition of length $p$ which contains a zero).

We find it convenient to represent partitions on James' abacus, see \cite[Section 2.7]{JamesKerberSymmetric81} for details. Let $\lambda=(\lambda_1,\lambda_2,\dots, \lambda_t)$ be a partition. Fix an integer $r$.
The \emph{$\beta$-set} of $\lambda$ (with respect to $r$) is the set
\[
B_r(\lambda)
=
\{\lambda_i-i+r \mid i \in \mathbb{N}\}
\]
where we take $\lambda_i$ to be zero for $i>t$; this is a convention we use throughout. We set  $\beta_i=\lambda_i-i+r$, so then $\beta_1>\beta_2>\dots$. The partition may be recovered from the $\beta$-set by letting $\lambda_i=\beta_i-(r-i)$.

Changing $r$ merely translates all elements of the $\beta$-set by a fixed integer. We call the set $B_0(\lambda)$ the \textit{canonical} $\beta$-set of $\lambda$; this will be the $\beta$-set we mostly work with, although for the sake of making the results more applicable, in many statements, we allow $r$ to be any integer.

Given a $\beta$-set $B_r(\lambda)$ we can represent $\lambda$ on an abacus with $p$ runners. Arrange the integers in $p$ vertical
columns (called \emph{runners}) according to residue modulo $p$:
\[
\begin{array}{cccc}
\vdots&\vdots&&\vdots\\
0 & 1 & \cdots & p-1 \\
p & p+1 & \cdots & 2p-1 \\
2p & 2p+1 & \cdots & 3p-1 \\
\vdots & \vdots & & \vdots
\end{array}
\]

The $p$-abacus display of $\lambda$ is obtained by placing a bead in
position $b$ whenever $b\in B_r(\lambda)$ and leaving all other
positions empty. Note that any $\beta$-set of a partition is bounded above and contains all integers less than or equal to some integer $l$. Thus, we can (and will) represent our abacus displays with a finite number of beads, always with the top row (and thus assumed every other row above that) with all $p$ beads present.

For an example consider $\lambda=(6,5,4,4,2,1,1)$ and $r=9$. Then:

\[B_9(\lambda)=\{\ldots, -3,-2,-1,0,1,3,4,6,9,10,12,14\}
\]
and the abacus display on $5$ runners is shown in Figure~\ref{fig:abacus-example}.
\begin{figure}[ht]
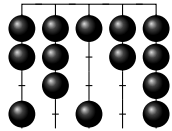

\label{figure:exampleabacus}
\centering
\[
\abacus(lmmmr,bbbbb,bbnbb,nbnnb,bnbnb) 
\]
\caption{The $5$-abacus display of $\lambda=(6,5,4,4,2,1,1)$ with $r=9$.}
\label{fig:abacus-example}
\end{figure}
One can easily read the parts of $\lambda$ from its abacus display by, for each bead, counting the number of empty lower-numbered positions (and discarding zeros).

Let $[\lambda] = \{ (i,j)\in \mathbb{N}^2 \mid j\leq \lambda_i\}$ be the \textit{Young diagram} of $\lambda$. We call its elements \textit{nodes}. Recall that a \textit{hook} of $\lambda$ is a set $h_{ij}(\lambda) = \{ (i,j')\in [\lambda] \mid j'\geq j \}\cup \{ (i',j)\in [\lambda] \mid i'\geq i \}$ for some $(i,j)\in [\lambda]$. A \textit{rim hook} of $\lambda$ is a set $r_{ij}(\lambda) = \{ (i',j')\in [\lambda] \mid i'\geq i, j'\geq j, (i'+1, j'+1)\notin [\lambda]\}$ for some $(i,j)\in [\lambda]$. It is a standard fact that hook $h_{ij}(\lambda)$ and the corresponding rim hook $r_{ij}(\lambda)$ have equal size. If this size is $p$, we say that the (rim) hook is a (rim) \textit{$p$-hook}, and if the size is divisible by $p$, we say it is a (rim) \textit{$p$-divisible hook}.

One can remove a rim hook of a partition $\lambda$ by removing it from $[\lambda]$, which produces a Young diagram of a new partition. The inverse procedure is called the addition of a rim hook. The addition and removal of rim $p$-hooks is simple on the abacus. Suppose a bead occupies position $b$ and the position $b-p$ immediately above it (on the same runner) is empty. Sliding the bead upward from $b$ to $b-p$ replaces $b$ by $b-p$ in the $\beta$-set. On the level of Young diagrams, such pairs correspond to (rim) $p$-hooks, and this operation removes the corresponding rim $p$-hook. An important property of this sliding is that it changes $\beta$-sets of partitions with respect to some $r$ to $\beta$-sets of partitions with respect to the same $r$.  

Iterating all possible upward bead slides produces the \textit{$p$-core} of
$\lambda$, and the total number of such slides is the \textit{$p$-weight} of
$\lambda$. A partition which is equal to its $p$-core is called a \textit{$p$-core partition}. It is clear that $\gamma$ is a $p$-core partition if and only if its $p$-weight is $0$, which is further equivalent to $\gamma$ having no $p$-hooks. See Figure~\ref{fig:hooks} for an example.

\begin{figure}[ht]
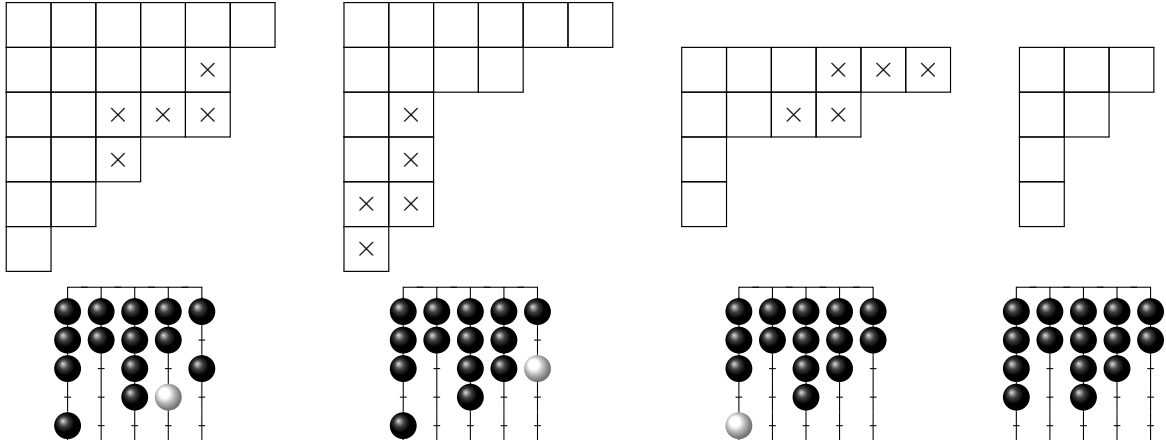

    \centering
    \ytableausetup{centertableaux}
    \begin{ytableau}
    {} & & & & & \\
    & & & & \times \\
    & & \times & \times & \times\\
    & & \times \\
    &\\
    \\
    \end{ytableau}
    \qquad
    \begin{ytableau}
    {} & & & & & \\
    & & & \\
    & \times \\
    & \times \\
    \times & \times \\
    \times \\
    \end{ytableau}
    \qquad
    \begin{ytableau}
    {} & & & \times & \times & \times \\
    & & \times & \times \\
    \\
    \\
    \end{ytableau}
    \qquad
    \begin{ytableau}
    {} & & \\
    & \\
    \\
    \\
    \end{ytableau}
    \[
    \qquad \abacus(lmmmr,bbbbb,bbbbn,bnbnb,nnbon,bnnnn) \qquad\qquad\qquad \abacus(lmmmr,bbbbb,bbbbn,bnbbo,nnbnn,bnnnn) \qquad\qquad\qquad \abacus(lmmmr,bbbbb,bbbbb,bnbbn,nnbnn,onnnn) \qquad\qquad \abacus(lmmmr,bbbbb,bbbbb,bnbbn,bnbnn,nnnnn)
    \]
    \caption{Let $p=5$. Starting with partition $(6,5^2,3,2,1)$ (on the left), we obtain its $p$-core $(3,2,1^2)$ (on the right) by successively removing three $p$-rim hooks labeled by nodes $(2,3)$, $(3,1)$ and $(1,3)$, and denoted by $\times$ in the diagrams. Its $p$-weight is therefore $3$. One can see that on the abacus, these removals correspond to sliding a bead (drawn white) one space upwards. The process then terminates when no bead has a vacant position above it.}
    \label{fig:hooks}
\end{figure}

There is nothing special about $p$; the above statements remain valid if $p$ is replaced by any positive integer, in particular, by any multiple of $p$. In general, (rim) hooks correspond to pairs $(f,b)$ with $f<b$ such that James' abacus has a bead at position $b$, while position $f$ is empty. The (rim) $p$-divisible hooks then correspond to such pairs with $f=b-ap$ for some positive integer $a$. The replacement of $b$ by $f$ in a $\beta$-set then corresponds to removing the corresponding rim hook.

We will need three more standard definitions.

\begin{definition}
For a hook $h=h_{ij}(\lambda)$ of $\lambda$ we define its \textit{arm length} as $a(h)=a_{ij}(\lambda)=\lambda_i-j$ and its \textit{leg length} as $\ell(h)=\ell_{ij}(\lambda)=\lambda'_j-i$, and the \textit{hook length} as $$|h|=1+a(h)+\ell(h).$$
\end{definition}

We make the same definitions for the corresponding rim hook. One can calculate these directly from the abacus.

\begin{lemma}
\label{lemma:readinglegandarmlength off abacus}
Let $\lambda$ be a partition and suppose that on its abacus there is a bead at position $b$ and a vacant position $f<b$, corresponding to a (rim) hook $h$. Then the leg length of $h$ is obtained by counting the occupied abacus positions strictly between $f$ and $b$, and the arm length is the number of empty positions in the same interval.
\end{lemma}

\begin{example}\label{ex:arm lengths}
    The arm lengths of the removed rim hooks in Figure~\ref{fig:hooks} are $2$, $1$ and $3$, respectively. This agrees with Lemma~\ref{lemma:readinglegandarmlength off abacus}: the number of empty positions strictly between the highlighted beads and the (empty) positions immediately above them are, again, $2$, $1$ and $3$. The leg lengths are $2$, $3$ and $1$, respectively.
    Note that the left-most partition $(6,5^2,3,2,1)$ has three (rim) $5$-divisible hooks (all are in fact (rim) $5$-hooks). They all have arm length $2$, and thus $\mu = (6,5^2,3,2,1)$ satisfies the condition in Theorem~\ref{theorem:even arms}.  
\end{example}

For a $\beta$-set $B$ of some partition $\lambda$ and integers $u\leq v$, we write $\emp_B(u,v)$ for the number of empty positions $f$ on the James' abacus of $B$ such that $u\leq f \leq v$. If $u-1\notin B$ and $v\in B$, then Lemma \ref{lemma:readinglegandarmlength off abacus} tells us that $\emp_B(u,v)$ is the arm length of the hook of $\lambda$ corresponding to $(u-1, v)$. We find it convenient to group such $B$, $u-1$ and $v$ into a triple $\mathcal{B}=(B,u-1,v)$ which we call a \textit{hook triple} and use the shorthand $\emp(\mathcal{B})$ for $\emp_B(u,v)$.

\begin{definition}\label{def: p-residues}
    The \textit{$p$-residue} of a node $(i,j)\in [\lambda]$ is the remainder of $j-i$ modulo $p$ (between $0$ and $p-1$). The \textit{$p$-content} of $\lambda$ is a composition $C(\lambda):=(r_0,r_1,\dots,r_{p-1})$ where $r_k$ is the number of nodes $(i,j)\in [\lambda]$ with $p$-residue $k$.  
\end{definition}

If $h$ is a (rim) $p$-hook, then the $p$-residues of nodes $(i,j)\in h$ are numbers $0,1,\dots, p-1$ in some order. Therefore, if $\lambda$ has $p$-core $\gamma$ and $p$-weight $w$, its $p$-content is $C(\lambda) = C(\gamma) + \bm{w}$, where $\bm{w}$ is a sequence consisting of $p$ copies of $w$. In other words, the $p$-core and $p$-weight determine the $p$-content. In fact, the converse is also true.

\begin{theorem}[\cite{JamesKerberSymmetric81}, Theorem~2.7.41]\label{thm: p-residues}
    Two partitions have the same $p$-content if and only if they have the same $p$-core and $p$-weight.
\end{theorem}

Examples of $p$-residues and $p$-contents are in Figure~\ref{fig:content}

\begin{figure}[ht]
    \centering
    \ytableausetup{centertableaux}
    \begin{ytableau}
    0 & 1 & 2 & 3 & 4 & 0 \\
    4 & 0 & 1 & 2 & 3 \\
    3 & 4 & 0 & 1 & 2 \\
    2 & 3 & 4 \\
    1 & 2\\
    0 \\
    \end{ytableau}
    \qquad
    \begin{ytableau}
    0 & 1 & 2 \\
    4 & 0 \\
    3\\
    2\\
    \end{ytableau}
    \caption{Let $p=5$. The $p$-residues of the nodes of the Young diagram of partition $\mu=(6,5^2,3,2,1)$ and its $p$-core $\gamma=(3,2,1^2)$ are displayed above. Their $p$-contents are $C(\mu) = (5, 4, 5, 4, 4)$ and $C(\gamma) = (2,1,2,1,1)$. Thus $C(\mu) = C(\gamma) + \bm{3}$, which agrees with Theorem~\ref{thm: p-residues} which states that the partitions with $p$-content $C(\gamma) + \bm{3}$ are precisely partitions with $p$-core $\gamma$ and $p$-weight $3$.}
    \label{fig:content}
\end{figure}

Finally, we recall the \textit{dominance order}, a partial order $\unrhd$ on the set of partitions defined by $\mu \unrhd \lambda$ if and only if for all integers $j\geq 1$ we have
\[
\sum_{i=1}^j \mu_i \geq \sum_{i=1}^j \lambda_i,
\]
where we again use the convention that $\lambda_i$ and $\mu_i$ is $0$ whenever $i$ is greater than the length of $\lambda$ and $\mu$, respectively. For readers unfamiliar with the dominance order, our arguments remain valid when the lexicographic order -- a refinement of the dominance order -- is used instead; however, we opted to use the standard order in the representation theory of symmetric groups. 

\begin{example}\label{ex:dominance}
    For $\lambda=(6,4,3^4)$, $\mu=(6,5^2,3,2,1)$ and $\nu=(6,5^2,2^2,1^2)$, the sums of their first $i$ parts with $i=1,2,\dots$ are $6,10,13,16,19,22,22,22,\dots$, $6,11,16,19,21,22,22,22,\dots$ and $6,11,16,18,20,21,22,22,\dots$, respectively. Thus $\mu\rhd\lambda$ and $\mu\rhd\nu$ and $\lambda$ and $\nu$ are incomparable in the dominance order. 
\end{example}

We will occasionally use the fact that if $\lambda$ and $\nu$ are partitions of the same size and there are rim hooks $g$ and $h$ of $\lambda$ and $\nu$, respectively, such that $[\lambda]\setminus g = [\nu] \setminus h$, then $\lambda$ and $\nu$ are comparable in the dominance order. Furthermore, working with $\beta$-sets with respect to fixed $r$, they are easy to compare: if $(B,b-s,b)$ and $(C,c-s,c)$ are hook triples corresponding to $g$ and $h$, respectively, then $\lambda\lhd\nu$ if and only if $b<c$.

\subsection{Representation theory of symmetric groups}\label{se:rep}

Let $F$ be a field of (prime) characteristic $p$ and $n$ be a non-negative integer. Here, we summarise basic results from the representation theory of the symmetric group $S_n$ in characteristic $p$. The statements can be found in standard references for modular representation theory of finite groups, such as \cite{AlperinLocal86, BensonRepresentationCohomologyI91, SerreRepresentation77} for representation theory of symmetric groups, such as \cite{JamesSymmetric78, JamesKerberSymmetric81}, and for decomposition numbers of symmetric groups (and Hecke algebras) \cite{DipperJamesHecke86, MathasHecke99}.

The \textit{Specht modules} $S^{\lambda}$ indexed by partitions of $n$ are modules of $S_n$ defined over integers, and thus over any ring by the extension of scalars. When defined over a field of characteristic zero, they are precisely the (pairwise non-isomorphic) irreducible modules; however, this is not the case for prime characteristic. Indeed, the irreducible $FS_n$-modules $D^{\mu}$ are indexed by \textit{$p$-regular} partitions of $n$ -- partitions of $n$ such that each of its entries appears at most $(p-1)$-times -- and can be defined as (simple) heads of the corresponding Specht modules.

\begin{example}\label{ex:regular}
    Partition $(6,5^2,3,2,1)$ is $p$-regular for any odd prime $p$, although it is not $2$-regular thanks to its two parts equal to $5$. Combining this with Example~\ref{ex:arm lengths}, we see that Theorem~\ref{theorem:even arms} applies with $p=5$ and $\mu=(6,5^2,3,2,1)$ (and $\lambda$ any partition of $22$).
\end{example}

The \textit{decomposition number} $d_{\lambda\mu} = \left[S^{\lambda} : D^{\mu}\right]$ counts the number of composition factors of the reduction of $S^{\lambda}$ modulo $p$ isomorphic to $D^{\mu}$. We will need some standard facts about decomposition numbers. Firstly, we have $d_{\mu\mu} = 1$. Secondly, if the decomposition number $d_{\lambda\mu}$ is positive, then $\mu\unrhd\lambda$. Thirdly, the decomposition numbers depend only on the characteristic $p$, not the actual field $F$. The next (key) fact requires more prerequisites.

Let $G$ be a finite group and $V$ an indecomposable $FG$-module. A minimal subgroup $K$ of $G$ (with respect to inclusion) for which there is an $FK$-module $U$ such that $V$ is a direct summand of $U\Ind^G$ is called a \textit{vertex} of $V$. Up to conjugation in $G$, there is a unique vertex of $V$, and it is a $p$-group. We call $V$ a \textit{projective} module if its vertex is the trivial subgroup of $G$ and, more generally, say that an $FG$-module is \textit{projective} if all its indecomposable summands are projective. This is a less standard definition of projective modules; however, it is most suitable for our purposes. Projective objects can be defined in a variety of other settings; we will occasionally need them in the setting of $\mathcal{O}$-free $\mathcal{O}G$-modules for a discrete valuation ring $\mathcal{O}$ of characteristic $0$. In this setting, $\mathcal{O}G$-modules are always assumed to be free over $\mathcal{O}$, even if it is not explicitly mentioned.

Each indecomposable projective $FG$-module has a simple head, and, working up to isomorphisms, passing to the (simple) head of a module is a bijection between the indecomposable projective $FG$-modules and the irreducible $FG$-modules. Returning to the symmetric group, we denote by $P^{\mu}$ the indecomposable projective module with simple head $D^{\mu}$. Let us now assume that $F$ is a residue field of a discrete valuation ring $\mathcal{O}$ of characteristic $0$. Up to isomorphisms, reduction modulo the maximal ideal of $\mathcal{O}$ is a bijection between (indecomposable) projective $\mathcal{O}S_n$-modules and (indecomposable) projective $FS_n$-modules; we write $\psi^{\mu}$ for the ordinary character of the unique projective lift of $P^{\mu}$ (which, for the sake of simplicity, we also denote by $P^{mu}$). The essential property of decomposition numbers is that if $\chi^{\lambda}$ denotes the ordinary character of the Specht module $S^{\lambda}$, then
\begin{equation}\label{eq:projective}
    \psi^{\mu} = \sum_{\lambda} d_{\lambda\mu} \chi^{\lambda},
\end{equation}
where the sum runs over all partitions of the same size as $\mu$.

For a finite group $G$, recall that the \textit{blocks} of $FG$ are the indecomposable summands of $FG$ considered as an $FG$-bimodule. An $FG$-module $V$ belongs to block $A$ if $AV = V$ while $A'V=0$ for any other block $A'$ of $FG$. If $V$ is any $FG$-module, we then have the canonical block decomposition $V= \bigoplus_A AV$, where $A$ runs over all blocks of $FG$ and each $AV$ lies in block $A$. The blocks of symmetric groups are well-known, even though the theorem describing them is misleadingly known as Nakayama's conjecture.

\begin{theorem}[Nakayama's conjecture]\label{thm:NakayamaF}
    The blocks of $FS_n$ correspond to $p$-core partitions $\gamma$ such that there exists a partition of $n$ with $p$-core $\gamma$. The $FS_n$-modules $S^{\lambda}$, $D^{\mu}$ and $P^{\mu}$ belong to the block corresponding to the $p$-core of $\lambda$, $\mu$ and $\mu$, respectively. 
\end{theorem}

By Theorem~\ref{thm: p-residues}, we can alternatively index the blocks of $FS_n$ by all the possible $p$-contents of partitions of $n$. The $FS_n$-modules $S^{\lambda}$, $D^{\mu}$ and $P^{\mu}$ then belong to the block indexed by $C(\lambda)$, $C(\mu)$ and $C(\mu)$, respectively. As with the projective modules, if $F$ is the residue field of a discrete valuation ring $\mathcal{O}$ of characteristic $0$, blocks of $FS_n$ lift to blocks of $\mathcal{O}S_n$. One can then still talk about $\mathcal{O}S_n$-modules belonging to a block and the canonical decomposition. There is also an analogue of Theorem~\ref{thm:NakayamaF}, which we state here using the $p$-content, for future references.

\begin{theorem}[Nakayama's conjecture for $\mathcal{O}S_n$]\label{thm:NakayamaO}
    The blocks of $\mathcal{O}S_n$ correspond to $p$-contents of partitions of $n$. The $\mathcal{O}S_n$-modules $S^{\lambda}$ and $P^{\mu}$ belong to the block corresponding to $C(\lambda)$ and $C(\mu)$, respectively.
\end{theorem}

If $\tau$ is an integer sequence of length $p$, and $V$ is an $FS_n$-module or an $\mathcal{O}S_n$-module, we let $V_{\tau}$ be $AV$ if there is a block $A$ corresponding to $\tau$ (that is, $\tau$ is a $p$-content of some partition of $n$) and $0$ otherwise. The canonical decomposition of $V$ can be rewritten as $V = \bigoplus_{\tau} V_{\tau}$, where $\tau$ runs over all integer sequences of length $p$. We end this summary with the description of the outer tensor product with the sign.

\begin{theorem}[Pieri's rule]\label{thm:Pieri}
    Suppose that $M$ is a set of some partitions of $n$ and $\sum_{\lambda\in M}\chi^{\lambda}$ is an ordinary character of some $\mathcal{O}S_n$-module $V$. Then $\left(V\boxtimes \sgn\right)\Ind^{S_{n+k}}$ has ordinary character
    \[
    \sum_{\lambda\in M}\sum_{\nu\in \Add_k(\lambda)}\chi^{\nu},
    \]
    where $\Add_k(\lambda)$ is the set of partitions obtained by adding $k$ nodes to pairwise different (possibly empty) rows of $\lambda$.
\end{theorem}

\section{Odd sequences}
\label{sec:odd}
It has been well-known for decades that the ordinary character of the twisted Foulkes module $H^{(2^m;k)}$ decomposes as one copy of each irreducible character $\chi^\mu$ where $\mu$ runs over partitions of $2m+k$ with exactly $k$ odd parts (see \cite{InglisRichardsonSaxlModel90} for a short proof of this fact), so there is a long history of sorting partitions by counting how many odd parts they have. This continued in the work of Giannelli--Wildon we described in Section~\ref{sec:intro}. We show that the odd sequence -- a refinement of this statistic, where we keep track not only of which parts $\lambda_i$ are odd, but also of which runner on the $p$-abacus the corresponding bead lies -- also arises when considering decompositions of these twisted Foulkes modules.

Our application of odd sequences will be to the representation theory in odd prime characteristic $p$, but the underlying combinatorics we discuss here works for any odd integer $p>1$ (and in many cases even for all integers $p>1$). So, if not stated otherwise, in this and the following section $p>1$ is any odd integer. We recall the definition of the odd sequence of a partition from the introduction.

\begin{definition}
\label{def:oddsequence}
The \emph{odd sequence} of a partition $\lambda$ is a composition $O(\lambda):=(n_0, n_1, \dots, n_{p-1})$ where $n_i$ is the number of parts $\lambda_j$ of $\lambda$ with $\lambda_j$ odd and the $p$-residue of the right-most node in row $j$ equal to $i$, i.e. $\lambda_j - j \equiv i \pmod{p}$.
\end{definition}

\begin{remark}
\label{remark: oddsequencecountsoddbeads}
If we represent $\lambda$ on a $p$-abacus with a multiple of $p$ drawn beads (for instance, using the canonical $\beta$-set of $\lambda$), then $n_i$ counts the number of \textit{odd beads} on runner $i$, where an odd bead is defined as a bead with an odd number of empty lower-numbered positions on the abacus.
\end{remark}

An example of an odd sequence is in Figure~\ref{fig:odd sequence}.

\begin{figure}[ht]
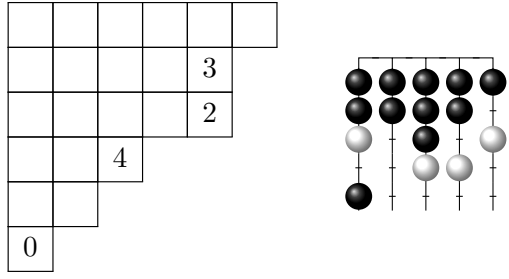

    \centering
    \ytableausetup{centertableaux}
    \begin{ytableau}
    {} &  &  &  &  &  \\
     &  &  &  & 3 \\
     &  &  &  & 2 \\
     &  & 4 \\
     & \\
    0 \\
    \end{ytableau}
    \qquad
    $\abacus(lmmmr,bbbbb,bbbbn,onbno,nnoon,bnnnn)$
    \caption{Let $p=5$. The right-most nodes of odd parts of $\mu=(6,5^2,3,2,1)$ contain their $p$-residue. One can immediately read off the odd sequence of $\mu$: it is $(1,0,1,1,1)$. Alternatively, one can count the number of odd beads (drawn white) on each runner of the displayed abacus of the canonical $\beta$-set of $\mu$.}
    \label{fig:odd sequence}
\end{figure}

Now fix a $p$-core $\gamma$ and a composition $\theta$ of length $p$. It is an easy exercise (proved in more generality in \cite[Corollary~2.11]{TurekMullineuxArxiv25}) that we can, starting with the abacus display for $\gamma$, slide beads down in such a way to ensure $n_i$ odd beads on runner $i$.
That is:

\begin{lemma}
\label{lem:omegaOiswelldefined}
  There exists a partition $\lambda$ with $p$-core $\gamma$ and odd sequence $O(\lambda)=\theta$.
\end{lemma}

Lemma \ref{lem:omegaOiswelldefined} is false for $p$ even, as adding a rim $p$-hook preserves the parity of the size of the odd sequence, so starting with the odd sequence of $\gamma$ we can only obtain odd sequences of the same size parity. The full characterization of $\theta$ for which Lemma~\ref{lem:omegaOiswelldefined} holds when $p$ is even can be read off from \cite[Corollary~2.11]{TurekMullineuxArxiv25}. For simplicity, we continue with $p$ odd.

Lemma~\ref{lem:omegaOiswelldefined} justifies the following definitions:

\begin{definition}
\label{def:omega_O}
   Let $w_\theta(\gamma)$ denote the minimum $p$-weight of a partition with $p$-core $\gamma$ and odd sequence $\theta$.
\end{definition}

\begin{definition} 
   Let $\mathcal{E}_\theta(\gamma)$ denote the set of all partitions with $p$-core $\gamma$, $p$-weight $w_\theta(\gamma)$ and odd sequence $\theta$. Observe that each of these is a partition of $n=|\gamma|+pw_\theta(\gamma).$
\end{definition}

\begin{remark}\label{remark:GWSets}
    Recall that Giannelli and Wildon define $\mathcal{E}_k(\theta)$ to be the set of partitions with $p$-core $\gamma$, $k$ odd parts and minimal possible $p$-weight, called $w_k(\theta)$. We immediately see that Giannelli--Wildon's $\mathcal{E}_0(\gamma)$ is the same as our $\mathcal{E}_{\bm{0}}(\gamma)$, and consists of minimal (with respect to the size) even partitions with $p$-core $\gamma$. More generally, partitions with odd sequence $\theta\vDash k$ have $k$ odd parts; therefore $w_{\theta}(\gamma)\geq w_k(\gamma)$ and we obtain equality if and only if there is a partition with odd sequence $\theta$ in $\mathcal{E}_k(\theta)$. In other words,
    \begin{equation}\label{eq:disjoint}
        \mathcal{E}_k(\gamma) = \bigsqcup_{\substack{\theta\vDash k \textnormal{ of length } p \\ w_{\theta}(\gamma) = w_k(\gamma)}} \mathcal{E}_{\theta}(\gamma).
    \end{equation}
\end{remark}

The sets $\mathcal{E}_{\theta}(\gamma)$ always have a unique maximal element in the dominance order; see \cite[Proposition~6.14]{TurekMullineuxArxiv25}. Furthermore, these elements come with an alternative description.

\begin{proposition}[\cite{TurekMullineuxArxiv25}, Proposition~1.6]\label{pr:equivalence max}
	Let $\mu$ be a partition with $p$-core $\gamma$ and odd sequence $\mathcal{\theta}$. Then all arm lengths of $p$-divisible hooks of $\mu$ are even if and only if $\mu$ is the maximal element of $\mathcal{E}_{\theta}(\gamma)$. In that case, $\mu$  is $p$-regular if and only if $\theta$ contains a zero.
\end{proposition}

\begin{example}\label{ex:max}
    Let $p=5$ and $\mu=(6,5^2,3,2,1)$. We saw in Example~\ref{ex:arm lengths} that all arm lengths of $p$-divisible hooks of $\mu$ are even. By Proposition~\ref{pr:equivalence max}, $\mu$ is the (unique) maximal element of $\mathcal{E}_{\theta}(\gamma)$, where $\theta = (1,0,1,1,1)$ is the odd sequence of $\mu$ and $\gamma=(3,2,1^2)$ is the $p$-core of $\mu$. Since the $p$-weight of $\mu$ is $3$, we have $w_{\theta}(\gamma)=3$. The whole set $\mathcal{E}_{\theta}(\gamma)$ is the first set in Example~\ref{example:p=5gamma=3211allcolumns}. Note that $\mu$ is $p$-regular and $\theta$ contains a zero; this agrees with the final part of Proposition~\ref{pr:equivalence max}.
\end{example}

We now temporarily return to the assumption that $p$ is an odd prime. Proposition~\ref{pr:equivalence max} (or more precisely, its final part) says that the first part of Theorem~\ref{theorem:maintheorem} is true. It also makes the fact that Theorem~\ref{theorem:even arms} and Theorem~\ref{theorem:maintheorem} concern the same columns of the decomposition matrices of symmetric groups apparent. We can therefore choose and prove only one of them; we choose Theorem~\ref{theorem:maintheorem}.

We can also use the final part of Proposition~\ref{pr:equivalence max} to properly justify our earlier claim from Section~\ref{sec:intro} that Theorem \ref{theorem:maintheorem} implies that the sets $\mathcal{X}_1,\mathcal{X}_2,\dots,\mathcal{X}_c$ in Theorem~\ref{theorem:GWThm1.1} are all of the form $\mathcal{E}_{\mathcal{\theta}}(\gamma)$. In fact, we show that they are the sets on the right-hand side of \eqref{eq:disjoint}.

\begin{lemma}\label{le:GWCorollary}
    The sets $\mathcal{X}_1,\mathcal{X}_2,\dots,\mathcal{X}_c$ in Theorem~\ref{theorem:GWThm1.1} are the sets on the right-hand side of \eqref{eq:disjoint}, that is, sets $\mathcal{E}_{\theta}(\gamma)$ with $\theta\vDash k$ of length $p$ such that $w_{\theta}(\gamma) = w_k(\gamma)$. Furthermore, all such $\theta$ contain a zero.
\end{lemma}

\begin{proof}
Suppose that we have already shown that $\mathcal{X}_1,\mathcal{X}_2,\dots,\mathcal{X}_t$ coincide with some of the sets on the right-hand side of \eqref{eq:disjoint}, and that all the corresponding $\theta$ contain a zero. We are done if $\mathcal{E}_k(\gamma) \setminus \bigsqcup_{i=1}^t \mathcal{X}_i$ is empty. Otherwise, take a maximal partition $\mu$ with respect to the dominance order in this set and without loss of generality assume that it lies in $\mathcal{X}_{t+1}$.

By Theorem~\ref{theorem:GWThm1.1}, $\mu\unlhd \nu_{t+1}$, where $\nu_{t+1}$ is the unique maximal element of $\mathcal{X}_{t+1}$ (which is $p$-regular), and our choice of $\mu$ forces $\mu = \nu_{t+1}$. By \eqref{eq:disjoint}, we know that $\mu\in \mathcal{E}_{\tilde{\theta}}(\gamma)\subseteq \mathcal{E}_k(\gamma) \setminus \bigsqcup_{i=1}^t \mathcal{X}_i$ for some $\tilde{\theta}\vDash k$ of length $p$ such that $w_{\tilde{\theta}}(\gamma) = w_k(\gamma)$. Again, the choice of $\mu$ forces $\mu$ to be the maximal element of $\mathcal{E}_{\tilde{\theta}}(\gamma)$ (with respect to the dominance order); thus, the final part of Proposition~\ref{pr:equivalence max} shows that $\theta$ contains a zero. Theorem~\ref{theorem:maintheorem} thus applies to $\tilde{\theta}$ and $\gamma$ and shows that the non-zero entries in the column of the decomposition matrix labeled by $\mu$ are in rows given by $\mathcal{E}_{\tilde{\theta}}(\gamma)$. Since $\mu = \nu_{t+1}$, Theorem~\ref{theorem:GWThm1.1} shows that these rows are also given by $\mathcal{X}_{t+1}$; hence $\mathcal{X}_{t+1} = \mathcal{E}_{\tilde{\theta}}(\gamma)$, as required.
\end{proof}

\begin{remark}\label{re:newcolumns}
    Note that when $w_{\theta}(\gamma) > w_k(\gamma)$ for some $\theta\vDash k$ of length $p$ which contains a zero, then the decomposition column from Theorem~\ref{theorem:maintheorem} does \emph{not} come from the columns of Giannelli--Wildon. We see examples of such columns in Section~\ref{subsection:Examples}.
\end{remark}

The odd sequence can be interpreted algebraically when $p$ is an odd prime. We start with a motivational example, before stating the full result. Recall that in our notation $C(\gamma) + \bm{w}$ is the $p$-content of partitions with $p$-core $\gamma$ and $p$-weight $w$. In the rest of this section, $F$ is a residue field of a discrete valuation ring $\mathcal{O}$, where $F$ has characteristic $p$ and $\mathcal{O}$ has characteristic $0$.

\begin{example}\label{ex:foulkes}
    Let $p=5$, $\tau = (4,4,4,3,3)$ and $\rho=(5,4,5,4,4)$. To motivate the next lemma, we compute the ordinary character of the $\mathcal{O}S_{22}$-module
    \[
    V=\left(\left(H^{(2^9)}_{\tau} \boxtimes \sgn \right)\Ind^{S_{22}}\right)_{\rho}.
    \]
    The ordinary character of $H^{(2^9)}_{\tau}$ is the sum of $\chi^{\lambda}$ over even partitions $\lambda$ with $p$-content $\tau=(4,4,4,3,3)$; there are four of them (drawn in green in Figure~\ref{fig:even}):
    \[
    (6,4^2,2^2), (6,4,2^4), (4^4,2), (4^2,2^5).
    \]
    By Pieri's rule (Theorem~\ref{thm:Pieri}), we obtain the ordinary character of $\left(H^{(2^9)}_{\tau} \boxtimes \sgn \right)\Ind^{S_{22}}$ by summing $\chi^{\lambda}$ over partitions $\lambda$ obtained from one of our four even partitions by adding a node to $4$ of its (distinct) rows. Finally, the ordinary character of $V$ is obtained in the same way, but we require the $4$ added nodes to have $p$-residues $0,2,3$ and $4$ (so the final partitions have $p$-content $\rho=(5,4,5,4,4) = \tau + (1,0,1,1,1)$). For our four even partitions, we can add these $4$ nodes in three, two, two and three ways, respectively; one way for each of them is in Figure~\ref{fig:even}. The $10$ obtained partitions are precisely the elements of $\mathcal{E}_{\theta}(\gamma)$ where $\gamma=(3,2,1^2)$ and $\theta=(1,0,1,1,1)$ (and $w_{\theta}(\gamma)=3$); see the first set in Example~\ref{example:p=5gamma=3211allcolumns}. Note that if we let $w=w_{\theta}(\gamma)$, Figure~\ref{fig:content} shows that $C(\gamma) +\bm{w} - \theta=\tau$ and $C(\gamma) +\bm{w}=\rho$.
\end{example}

\begin{figure}[ht]
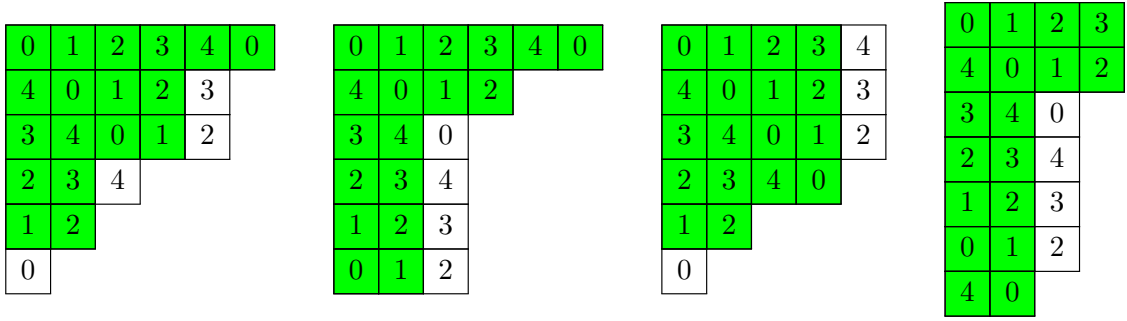

    \centering
    \ytableausetup{centertableaux}
    \ytableaushort{012340,40123,34012,234,12,0}
*[*(green)]{6,4,4,2,2}
\qquad
\ytableaushort{012340,4012,340,234,123,012}
*[*(green)]{6,4,2,2,2,2}
\qquad
\ytableaushort{01234,40123,34012,2340,12,0}
*[*(green)]{4,4,4,4,2}
\qquad
\ytableaushort{0123,4012,340,234,123,012,40}
*[*(green)]{4,4,2,2,2,2,2}
    \caption{Even partitions with $5$-content $(4,4,4,3,3)$ (drawn green), each with $4$ added nodes in different rows with $5$-residues $0,2,3$ and $4$.}
    \label{fig:even}
\end{figure}

\begin{lemma}\label{le:FoulkesModules}
    Let $\gamma$ be a $p$-core partition, $\theta$ an integer sequence of length $p$ with a non-negative sum of entries $k\geq 0$ and $w$ a non-negative integer such that $m = (|\gamma|+pw-k)/2$ is a non-negative integer. The ordinary character of the $\mathcal{O}S_{2m+k}$-module $\left(\left(H^{(2^m)}_{C(\gamma) +\bm{w} - \theta} \boxtimes \sgn \right)\Ind^{S_{2m+k}}\right)_{C(\gamma) + \bm{w}}$ is the sum of $\chi^{\lambda}$ labeled by partitions $\lambda$ with $p$-core $\gamma$, $p$-weight $w$ and odd sequence $\theta$. In particular, this is zero if $\theta$ has a negative entry.
\end{lemma}

    \begin{proof}
        We build the desired character from smaller pieces. The ordinary character of $H^{(2^m)}$ is a sum of $\chi^{\lambda}$ over even partitions $\lambda\vdash 2m$. Theorem~\ref{thm:NakayamaO} shows that if we pass to the ordinary character of $H^{(2^m)}_{C(\gamma)+\bm{w} - \theta}$ we only consider those $\lambda$ with $p$-content $C(\gamma)+\bm{w} - \theta$. We can apply Pieri's rule (Theorem~\ref{thm:Pieri}) to compute the ordinary character of $\left(H^{(2^m)}_{C(\gamma)+\bm{w} - \theta} \boxtimes \sgn \right)\Ind^{S_{2m+k}}$; it is the sum of $\chi^{\lambda}$ with $\lambda$ a partition of size $2m+k$ with precisely $k$ odd parts $\lambda_j$ such that the $p$-content of $\lambda$ without the right-most nodes in its odd-length rows is $C(\gamma)+\bm{w} - \theta$.

        Finally, passing to the block labeled by $C(\gamma)+\bm{w}$, we restrict these $\lambda$ to those of $p$-content $C(\gamma) + \bm{w}$. In other words, the desired ordinary character is the sum of $\chi^{\lambda}$ over $\lambda$ such that
        \begin{enumerate}
            \item $\lambda$ has size $2m+k$ and precisely $k$ odd parts $\lambda_j$; and
            \item $\lambda$ has $p$-content $C(\gamma)+\bm{w}$; and
            \item $\lambda$ has odd sequence $(C(\gamma)+\bm{w})-(C(\gamma)+\bm{w} - \theta) = \theta$.
        \end{enumerate}
        The second condition can be rephrased as `$\lambda$ has $p$-core $\gamma$ and $p$-weight $w$'. Thus, these conditions match the desired description if we show that the first condition is implicit from the remaining two. By our assumption $|\gamma|+pw = 2m+k$, and thus any partition with $p$-core $\gamma$ and $p$-weight $w$ has size $2m+k$. Furthermore, any partition with odd sequence $\theta$ has $k=|\theta|$ odd parts; thus, the first condition is redundant, and the result follows.
    \end{proof}

We deduce an alternative description of the set $\mathcal{E}_{\theta}(\gamma)$.

\begin{corollary}\label{cor:FoulkesModules}
    Let $\gamma$ be a $p$-core partition, $\theta$ a composition of length $p$ of size $k$ and $w=w_{\theta}(\gamma)$. Then $m = (|\gamma|+pw-k)/2$ is a non-negative integer, and the ordinary character of the $\mathcal{O}S_{2m+k}$-module $\left(\left(H^{(2^m)}_{C(\gamma) +\bm{w} - \theta} \boxtimes \sgn \right)\Ind^{S_{2m+k}}\right)_{C(\gamma) +\bm{w}}$ is
    \[
    \sum_{\lambda\in \mathcal{E}_{\theta}(\gamma)}\chi^{\lambda}.
    \]
\end{corollary}

\begin{proof}
    We only need to show that $m = (|\gamma|+pw-k)/2$ is a non-negative integer; the rest of the statement will immediately follow from Lemma~\ref{le:FoulkesModules}. From the definition of $w=w_{\theta}(\gamma)$, there is a partition $\lambda$ of size $|\gamma|+pw$ with odd sequence $\theta$. In turn, $\lambda$ has exactly $k = |\theta|$ odd parts, and their removal results in an even partition of size $|\gamma|+pw-k$. The result follows.
\end{proof}

\subsection{Examples}
\label{subsection:Examples}
Here are a couple of applications of Theorem~\ref{theorem:maintheorem} (and Theorem~\ref{theorem:even arms}). To distinguish between partitions and odd sequences, here, we use square brackets for the latter. Suppose $\gamma=(3,1)$ and $p=3$. We can easily compute that:
$$\mathcal{E}_1(\gamma)=\{(6,1),(3,2^2)\}.$$
Notice that both those partitions have odd sequence $[0,0,1] \vDash 1$; thus $\mathcal{E}_1(\gamma) = \mathcal{E}_{[0,0,1]}(\gamma)$ as in \eqref{eq:disjoint}. Theorem~\ref{theorem:maintheorem} as well as Giannelli--Wildon's result (Theorem~\ref{theorem:GWThm1.1}) gives us a column in the decomposition matrix of $S_7$ with two entries equal to one. The remaining two odd sequences $\theta\vDash 1$, namely $[1,0,0]$ and $[0,1,0]$ only show up for larger values of $n$, and thus fall out of the scope of Giannelli--Wildon's result, and require an application of Theorem~\ref{theorem:maintheorem}. For example:

\begin{example}
    \label{example:742oddseq100e=3} Let $p=3$, $\gamma=(3,1)$ and $\theta=[1,0,0]$. Then:

    $$\mathcal{E}_\theta(\gamma)=\{(7,4,2),(6,5,2),(6,4,3),(6,4,2,1)\}$$
gives the nonzero entries (all $1$) in the column labeled by partition $(7,4,2)$ in the decomposition matrix of $S_{13}$, which one can confirm in \cite[p.~150]{JamesSymmetric78}.
\end{example}

Suppose instead we wish to apply Theorem~\ref{theorem:maintheorem}, where we are given a particular block and want to produce a column in the decomposition matrix.

\begin{example}
\label{example:p=5gamma=3211allcolumns} Suppose $p=5$ and $\gamma=(3,2,1^2)$ and we are interested in the block of $FS_{22}$ with $p$-core $\gamma$ and $p$-weight $3$. In accordance with \cite[Theorem~5.1]{HemmerTurekArms26}, there are four $p$-regular partitions with all arm lengths of $p$-divisible hooks even, to which Theorem~\ref{theorem:even arms} applies:
$$(6,5^2,3,2,1), (13,7,1^2), (11,4,3,2^2), (9,4,3^2,2,1). $$
Thus, we have four columns of the decomposition matrix of $S_{22}$ corresponding to these $p$-regular partitions. Their odd sequences $\theta$ (which necessary satisfy $w_\theta(\gamma)=3$), and the corresponding $\mathcal{E}_\theta(\gamma)$ are displayed below.

\begin{align*}
    \mathcal{E}_{[1,\,0,\,1,\,1,\,1]}(\gamma) =\; &\bigl\{
(6,5^2,3,2,1),\;
(6,5^2,2^2,1^2),\;
(6,4^2,2^2,1^4),\;
(6,4,3^4),\;
(6,4,3,2^3,1^3),\;
(5^3,4,2,1),\\
&(5,4,3,2^4,1^2),\;
(4^4,2,1^4),\;
(4^2,3^4,2),\;
(4^2,3^2,2^3,1^2)
\bigr\}\\
\\
\mathcal{E}_{[1,\,0,\,2,\,1,\,0]}(\gamma) =\; &\bigl\{
(13,7,1^2),\;
(13,5,3,1),\;
(11,9,1^2),\;
(11,5^2,1),\;
(9^2,3,1),\;
(9,7,5,1)
\bigr\}\\
\\
\mathcal{E}_{[2,\,0,\,0,\,0,\,0]}(\gamma) =\; &\bigl\{
(11,4,3,2^2),\;
(11,4,2^3,1),\;
(10,4,3,2^2,1),\;
(8,7,3,2^2),\;
(8,7,2^3,1),\;
(8,6,3,2^2,1)
\bigr\}\\
\\
\mathcal{E}_{[2,\,0,\,0,\,1,\,1]}(\gamma) =\; &\bigl\{
(9,4,3^2,2,1),\;
(9,4,3,2^2,1^2),\;
(8,5,3^2,2,1),\;
(8,5,3,2^2,1^2),\;
(8,4,3^3,1),\\
&(8,4,3,2^2,1^3)
\bigr\}
\end{align*}

We can compute that $w_2(\gamma)=1<3=w_{[2,0,0,0,0]}(\gamma)$, so, as commented in Remark~\ref{re:newcolumns} the column labeled by $(11,4,3,2^2)$ does not arise as $\mathcal{X}_j$ in any of Giannelli--Wildon's columns.
\end{example}

We end with a substantial example of Lemma~\ref{le:GWCorollary}.
\begin{example}
    \label{example:GWambiguityresolvedbyourresult}
Let $p=7$ and $\gamma=(3,1)$. Then $w_4(\gamma)=2$. We compute that $\mathcal{E}_4(\gamma)$ consists of twelve partitions of $18$ in the block with $p$-weight $2$ and $p$-core $\gamma$. They are:
\begin{alignat*}{3}
\lambda_1 &= (10,3,2,1^3) &\qquad
\lambda_5 &= (7,5^2,1) &\qquad
\lambda_9 &= (4^2,3^2,2,1^2) \\[6pt]
\lambda_2 &= (10,2^2,1^4) &\qquad
\lambda_6 &= (7,4^2,1^3) &\qquad
\lambda_{10} &= (4^2,2^3,1^4) \\[6pt]
\lambda_3 &= (9,4,2,1^3) &\qquad
\lambda_7 &= (5^3,3) &\qquad
\lambda_{11} &= (3^4,2^3) \\[6pt]
\lambda_4 &= (8,4,2,1^4) &\qquad
\lambda_8 &= (5,4^2,3,1^2) &\qquad
\lambda_{12} &= (3^2,2^5,1^2)
\end{alignat*}
We compute the first four have odd sequence $\theta_1=[0,1,1,1,1,0,0]$, the second four have odd sequence $\theta_2=[0,0,1,1,1,0,1]$ and the final four have odd sequence $\theta_3=[1,1,1,0,0,0,1]$. Thus $\mathcal{E}_4(\gamma) = \mathcal{E}_{\theta_1}(\gamma) \sqcup \mathcal{E}_{\theta_2}(\gamma) \sqcup \mathcal{E}_{\theta_3}(\gamma)$ is the set partition asserted by Giannelli--Wildon (in Theorem~\ref{theorem:GWThm1.1}) and described by Lemma~\ref{le:GWCorollary}; and thus we actually obtain three columns in the decomposition matrix of $S_{18}$ labeled by $(10,3,2,1^3), (7,5^2,1)$ and $(4^2,3^2,2,1^2)$. One can confirm these match the known decomposition numbers for blocks of $p$-weight two \cite{RichardsDecomposition96}.
\end{example}

\section{Algorithms A1 and A2}

\label{section:Algorithms} As mentioned earlier, the ordinary characters of twisted Foulkes modules $H^{(2^m; k)}$ are multiplicity-free, which means, by \eqref{eq:projective}, that any indecomposable projective summands they contain must correspond to columns of the decomposition matrix containing only zeros and ones. A crucial tool in our proof will be the Jantzen--Schaper formula, which can give bounds on decomposition numbers. This formula is especially powerful with the additional information that decomposition numbers in a chosen column are all zero or one. In this section, we collect some combinatorial information about certain bead moves on the abacus that let us move between any two partitions in the set $\mathcal{E}_\theta(\gamma)$. Much of this information is taken from \cite[Section~6]{TurekMullineuxArxiv25} where a more general approach is taken. In particular, we use the algorithms A1 and A2 defined there, restricting to parameter $d=2$ (with a slight alteration in A1). Throughout this section, $p>1$ is an odd integer.

Suppose that $B\subseteq \mathbb{Z}$. We say a \textit{swap of $i$ and $j$ of $B$} changes $B$ to a new set by replacing $i$ and $j$ by $j$ and $i$, respectively. If both or none of $i$ and $j$ lie in $B$, then $B$ does not change. If only one, say $i$, lies in $B$, then $B$ is replaced by $\left( B \setminus\left\lbrace i\right\rbrace\right)  \cup \left\lbrace j\right\rbrace  $. In such a case, if $i>j$, we call this an \textit{up-move}, otherwise, we say it is a \textit{down-move}. We refer to up- and down-moves as \textit{proper moves}. The words `up' and `down' are chosen to match the corresponding slides of a bead on the abacus with $|i-j|$ runners, where $B$ is a $\beta$-set.
 	
Observe that if $|i-j|=ap$ for some positive integer $a$ and $B=B_r(\lambda)$ for a partition $\lambda$, then the swap of $i$ and $j$ of $B$ results in a set of the form $B_r(\mu)$ for some partition $\mu$ with the same $p$-core as $\lambda$. Moreover, in such a case, the $p$-weight of $\mu$ is the $p$-weight of $\lambda$ \emph{decreased by $a$} if the swap was an up-move, the $p$-weight of $\lambda$ \emph{increased by $a$} if the swap was a down-move, and the $p$-weight of $\lambda$ otherwise.
 	
Recall that $\mathcal{B}=(B,u,v)$ is a hook triple if $B$ is a $\beta$-set of some partition $\lambda$ and integers $u<v$ are such that $u\notin B$ and $v\in B$. We can now define the first algorithm A1. Let $\mathcal{B}=(B,b-ap,b)$ be a hook triple (so $a>0$). The algorithm A1 applied to $\mathcal{B}$ proceeds as follows. In turns, for $i= 0,1,2,\dots$ we swap $b+i-ap$ and $b+i$ of $B$ and terminate the process after two proper moves. So at step $i=0$ we do a proper swap of $b-ap$ and $b$, and then we may or may not do a second proper move in higher numbered positions.

The algorithm A2 is similar. We again start with a hook triple $\mathcal{B}=(B,b-ap,b)$. The algorithm A2 applied to $\mathcal{B} =(B,b-ap,b)$ proceeds as follows. In turns, for $i= 0,1,2,\dots$ we swap $b-i-ap$ and $b-i$ of $B$ and terminate the process after two proper moves. So again at step $i=0$ we swap $b-ap$ and $b$, and then we may or may not do a second proper move in lower numbered positions.

If the algorithm A1 (or A2) terminates, then, since only two bead moves were made, the final set is a $\beta$-set of some partition $\mu$ with the same $p$-core as the underlying partition $\lambda$ of the initial $\beta$-set $B$. Moreover, if the final swap of the algorithm is an up-move, then the $p$-weight of $\mu$ is less than the $p$-weight of $\lambda$ by $2a$. If the final swap is a down move, the $p$-weights are the same.

In terms of the Young diagram, both algorithms either remove two rim hooks of size $ap$ (if the final swap is an up-move) or remove one rim hook of size $ap$ and add it in a new place (if the final swap is a down-move). An example of the algorithms is in Figure~\ref{fig:algorithms}.

\begin{figure}[ht]
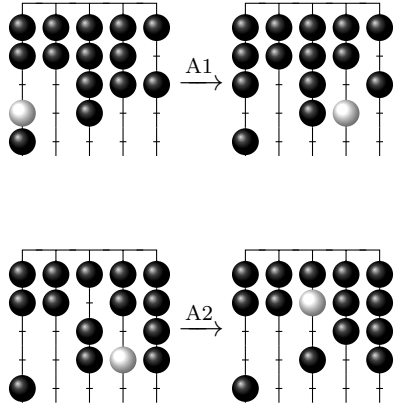

    \centering
    \[
     \abacus(lmmmr,bbbbb,bbbbn,nnbbb,onbnn,bnnnn) \xrightarrow{\textnormal{A1}} \abacus(lmmmr,bbbbb,bbbbn,bnbnb,nnbon,bnnnn)
    \]
    \vspace{1cm}
    \[
     \abacus(lmmmr,bbbbb,bbnbb,nnbnb,nnbob,bnnnn) \xrightarrow{\textnormal{A2}} \abacus(lmmmr,bbbbb,bbobb,nnnbb,nnbnb,bnnnn)
    \]
    \caption{Let $p=5$. The diagram shows the effects of algorithms A1 and A2 applied to hook triples $\mathcal{B} =(B,b-ap,b)$, where $B$ is a $\beta$-set of partition $(6,4,3^4)$ and $(6^4,4,3,1^2)$ (drawn on the left), respectively, $b$ is the position of the highlighted bead and $b-ap$ is immediately above it (so $a=1$). The resulting partitions are $(6,5^2,3,2,1)$ and $(6^2,5,3^2)$, respectively. In the first case, A1 performs one up-move and one down-move, while in the second case, A2 performs two up-moves (the highlighted beads are those that move). Observe that, applying A2 to $\mathcal{B} =(B,b-ap,b)$, where $B$ is the $\beta$-set on the right on the first line, $b$ is the position of highlighted bead and $b-ap$ is immediately above it (so $a=1$), has the inverse effect to A1.}
    \label{fig:algorithms}
\end{figure}

 The following is a helpful observation. We write $B_i$ for the state of the set in the algorithm A1 or A2 just before the swap of $b\pm i -ap$ and $b \pm i$, and, if $t$ is the number of swaps (assuming that the algorithm terminates), we write $B_t$ for the final set (reached after the swap of $b\pm (t-1) -ap$ and $b \pm (t-1)$).  

 	\begin{lemma}\label{le:A1emp}
 		Let $\mathcal{B}=(B,b-ap,b)$ be a hook triple. Then
        \begin{enumerate}
            \item $\emp_{B_i}(b+ i-ap,b+ i-1)$ does not change throughout the algorithm A1 applied to $\mathcal{B}$.
            \item $\emp_{B_i}(b-i-ap+1,b-i)$ does not change throughout the algorithm A2 applied to $\mathcal{B}$.
        \end{enumerate}
 	\end{lemma}
 	
 	\begin{proof}
 		We focus on (1); the case of (2) is similar. Let us consider how $\emp_{B_i}(b+i-ap,b+i-1)$ and $\emp_{B_{i+1}}(b+i+1-ap,b+i)$ differ. As sets, we obtain $\left\lbrace z \in \mathbb{Z}: b+i+1-ap\leq z\leq b+i \right\rbrace$ from $\left\lbrace z \in \mathbb{Z}: b+i-ap\leq z\leq b+i-1 \right\rbrace$ by removing $b+i-ap$ and adding $b+i$. Since $B_{i+1}$ is obtained from $B_i$ by swapping $b+i-ap$ and $b+i$, we immediately conclude that $\emp_{B_i}(b+i-ap,b+i-1)=\emp_{B_{i+1}}(b+i+1-ap,b+i)$. The result follows. 
 	\end{proof}

For a hook triple $\mathcal{B} = (B,u,v)$, recall the shorthand $\emp(\mathcal{B})$ for $\emp_B(u+1,v)$ which is the arm length of the hook of $\lambda$ corresponding to the pair $(u,v)$ (see Lemma~\ref{lemma:readinglegandarmlength off abacus}).

In \cite{TurekMullineuxArxiv25}, algorithm A1 only swaps $b$ and $b-ap$ if there is no bead between $b-ap$ and $b-1$. However, this situation cannot occur in our applications, where A1 is applied to a hook triple $\mathcal{B}=(B,b-ap,b)$ where $B$ is a $\beta$-set of some $\lambda\in\mathcal{E}_{\theta}(\gamma)$ and $\emp(\mathcal{B})$ is odd; indeed, if there was no bead between $b-ap$ and $b-1$, then swapping $b$ and $b-ap$ would preserve the $p$-core and odd sequence but decrease the $p$-weight, contradicting $\lambda\in\mathcal{E}_{\theta}(\gamma)$. Thus, we can use the following results from \cite{TurekMullineuxArxiv25}.
 	
\begin{lemma}\cite[Lemma~6.4 and 6.5]{TurekMullineuxArxiv25}\label{le:A1runner}
	Let $\mathcal{B}=(B,b-ap,b)$ be a hook triple such that $B$ is the canonical $\beta$-set of a partition $\lambda$ lying in some set $\mathcal{E}_{\theta}(\gamma)$ and $\emp(\mathcal{B})$ is odd. The algorithm A1 applied to $\mathcal{B}$ terminates and returns the canonical $\beta$-set of a partition $\mu\in\mathcal{E}_{\theta}(\gamma)$ such that $\mu\rhd \lambda$.
\end{lemma}

\begin{lemma}\cite[Lemma~6.7 and 6.8]{TurekMullineuxArxiv25}\label{le:A2runner}
	Let $\mathcal{B}=(B,b-ap,b)$ be a hook triple such that $B$ is the canonical $\beta$-set of a partition $\lambda$ lying in some set $\mathcal{E}_{\theta}(\gamma)$ such that $\theta$ contains a zero, and $\emp(\mathcal{B})$ is even. The algorithm A2 applied to $\mathcal{B}$ terminates and returns the canonical $\beta$-set of a partition $\mu\in\mathcal{E}_{\theta}(\gamma)$ such that $\mu\lhd \lambda$.
\end{lemma}

In the setting of Lemma~\ref{le:A1runner} and Lemma~\ref{le:A2runner}, we can conclude that the algorithm in question performs two proper moves: the initial up-move (with $i=0$) and the final down-move. We introduce some more notations for these cases. We write $\Aone(\mathcal{B})$ and $\Atwo(\mathcal{B})$ for the resulting $\beta$-sets after the algorithms A1 and A2, respectively, are applied to $\mathcal{B}$. Suppose the final proper move of the algorithm (a down-move) swaps $c-ap$ and $c$. Define hook triples $\widetilde{\Aone}(\mathcal{B})$ as $(\Aone(\mathcal{B}), c-ap, c)$ for algorithm A1 and $\widetilde{\Atwo}(\mathcal{B})$ as $(\Atwo(\mathcal{B}), c-ap, c)$ for algorithm A2. So $\widetilde{\Aone}$ and $\widetilde{\Atwo}$ keep track of the additional information regarding where the second proper swap was made.

In Figure~\ref{fig:algorithms}, if $\mathcal{B} = (B,b-ap,b)$ is the hook triple given by the top left $\beta$-sets $B$, with $b$ the position of its highlighted bead and $a=1$, and $\mathcal{C}$ is defined analogously for the top right abacus, then $\widetilde{\Aone}(\mathcal{B}) = \mathcal{C}$. As observed in the figure, we also have $\widetilde{\Atwo}(\mathcal{C}) = \mathcal{B}$. 
 	
With this notation, we immediately conclude the following from Lemma~\ref{le:A1emp}:
 	
 	\begin{corollary}\label{cor:A1newPair}
        Let $\mathcal{B}=(B,b-ap,b)$ be a hook triple with $B$ the canonical $\beta$-set of a partition $\lambda$ lying in some set $\mathcal{E}_{\theta}(\gamma)$.
        \begin{enumerate}
            \item Suppose that $\emp(\mathcal{B})$ is odd. Then $\emp\left(\widetilde{\Aone}(\mathcal{B})\right)$ is even.
            \item Suppose that $\theta$ contains a zero and $\emp(\mathcal{B})$ is even. Then $\emp\left(\widetilde{\Atwo}(\mathcal{B})\right)$ is odd.
        \end{enumerate}
 	\end{corollary}
 	
 	\begin{proof}
 		Lemma~\ref{le:A1runner} and \ref{le:A2runner} tell us that both algorithms terminate with a down-move. Observe that in (1) $\emp_B(b-ap, b-1)$ is even and if $t$ is the number of swaps, then $\widetilde{\Aone}(\mathcal{B}) = (\Aone(\mathcal{B}), b+t-ap-1, b+t-1)$, and so $\emp\left(\widetilde{\Aone}(\mathcal{B})\right) = \emp_{B_t}(b+t-ap, b+t-1)$ where $B_t$ is the terminating $\beta$-set. The result follows from Lemma~\ref{le:A1emp}(1) applied with $i=0$ and $i=t$.
        
        The situation in (2) is similar: $\emp_B(b-ap+1, b)$ is even by our assumption and $\emp\left(\widetilde{\Atwo}(\mathcal{B})\right) = \emp_{B_t}(b-t-ap+2, b-t+1) = \emp_{B_t}(b-t-ap+1, b-t)-1$, since the swap of $b-t-ap+1$ and $b-t+1$ is a down-move. The result now follows from Lemma~\ref{le:A1emp}(ii) applied with $i=0$ and $i=t$.
 	\end{proof}

We now move to formulating the relation between A1 and A2. Let $\gamma$ be a $p$-core partition and $\theta$ a composition of length $p$ which contains a zero. Let $T^o_{\theta}(\gamma)$ be the set of hook triples $\mathcal{B}=(B,b-ap,b)$ such that $B$ is the canonical $\beta$-set of some $\lambda\in \mathcal{E}_{\theta}(\gamma)$ and $\emp(\mathcal{B})$ is odd. Replacing odd by even, we obtain a set which we call $T^e_{\theta}(\gamma)$.
	
\begin{lemma}\label{le:inverse}
	The map $\widetilde{\Aone}$ is a bijection from $T^o_{\theta}(\gamma)$ to $T^e_{\theta}(\gamma)$ with the inverse given by $\widetilde{\Atwo}$.
\end{lemma}
	
\begin{proof}
	By Lemma~\ref{le:A1runner} and Corollary~\ref{cor:A1newPair}(1), the image of $\widetilde{\Aone}$ (applied to elements of $T^o_{\theta}(\gamma)$) lies inside $T^e_{\theta}(\gamma)$. Similarly, the image of $\widetilde{\Atwo}$ (applied to elements of $T^e_{\theta}(\gamma)$) lies inside $T^o_{\theta}(\gamma)$ using Lemma~\ref{le:A2runner} and Corollary~\ref{cor:A1newPair}(2). If $\mathcal{B}=(B,b-ap,b)\in T^o_{\theta}(\gamma)$, then the algorithm A2 applied to $\widetilde{\Aone}(\mathcal{B})$ performs the swaps of A1 applied to $\mathcal{B}$ in reverse order. As A1 performed only two proper moves, its first and its final swap, it follows that after the initial swap of A2, the next proper move must be the swap of $b$ and $b-ap$; thus $\widetilde{\Atwo}\left(\widetilde{\Aone}(\mathcal{B})\right)=\mathcal{B}$. One analogously verifies that the composition in the other direction is also the identity, which finishes the proof.
\end{proof}

\section{Jantzen--Schaper formula in the multiplicity-free setting}
\label{section: JantzenSchaper}

For the rest of the paper $p$ is an odd prime.

The Jantzen--Schaper formula provides a powerful tool for constraining the decomposition
numbers~$d_{\lambda\mu}$ of the symmetric group in characteristic~$p$. In its original form, it expresses a certain non-negative integer combination $\sum_{i \geq 1} \left[S^\lambda_i : D^\mu\right]$ as an explicit sum, computable from data in the Young diagram of~$\lambda$, over
contributions from partitions related to~$\lambda$ by removing and adding back rim $ap$-hooks. It provides an upper bound on the
decomposition number~$d_{\lambda\mu}$ using the decomposition numbers~$d_{\nu\mu}$ with $\nu\rhd\lambda$. The formula becomes especially powerful when it is known \emph{a
priori} that for fixed partition $\mu$ all decomposition numbers satisfy $d_{\lambda\mu} \leq 1$.  In
such cases the upper bound provided by the
Jantzen--Schaper formula determines $d_{\lambda\mu}$: one need not appeal to any additional
technique to pin down the exact value.

We state the version of the Jantzen--Schaper formula as presented in \cite[Theorem~1.6]{FayersWeightThree08} (in the case $e=p$, that is, our decomposition numbers are with respect to the symmetric group over a field of characteristic $p$). Let $v_p$ be the \textit{$p$-valuation} defined on positive integers by
	
	\begin{align*}
		v_p(x) = \begin{cases}
			0 & \textnormal{ if } p\nmid x;\\
			1+v_p(x/p) & \textnormal{ otherwise}.
		\end{cases}
	\end{align*}

One should think of the $p$-valuation function as a function which returns a positive integer if its input is divisible by $p$ and zero otherwise -- this is the only property of $v_p$ we need in this section. This allows one to generalize the results of this section to decomposition numbers of Hecke algebras of odd quantum characteristic $p$ (not necessarily a prime) and arbitrary characteristic, by replacing $v_p$ in the Jantzen--Schaper formula with a different function with this property; see \cite[Theorem~1.6]{FayersWeightThree08}. To follow the narrative of this paper, we will work in the setting of symmetric groups.

For partitions $\lambda$ and $\nu$ of equal size, let $\mathcal{H}(\lambda, \nu)$ be the set of pairs $(g,h)$, where $g$ is a rim hook of $[\lambda]$, $h$ is a rim hook of $[\nu]$ and $[\lambda]\setminus g = [\nu]\setminus h$. The quantity $c_{\lambda\nu}$ is then defined as
	\[
		c_{\lambda\nu} = \sum_{(g,h)\in \mathcal{H}(\lambda, \nu)} (-1)^{a(g) + a(h) + 1} v_p(|g|).
	\]
    In \cite[Theorem~1.6]{FayersWeightThree08} the arm lengths are replaced by leg lengths, but this makes no difference when $|\lambda|=|\nu|$. We further replace the lexicographical order by the dominance order; again, this does not change the statement, since whenever $c_{\lambda\nu}\neq 0$, and in particular, $\mathcal{H}(\lambda, \nu)$ is non-empty, then $\lambda$ and $\nu$ are comparable in the dominance order (see the fact after Example~\ref{ex:dominance}) and so one does not lose any relation by passing from the lexicographical to the dominance order.
	
	\begin{theorem}[Jantzen--Shaper formula]\label{th:JS}
		Suppose that $\lambda\neq\mu$ are partitions of the same size, with the latter being $p$-regular, and define
		\[
		n_{\lambda\mu} = \sum_{\nu \rhd \lambda} c_{\lambda\nu} d_{\nu\mu} .
		\]
		Then
		\[
		d_{\lambda\mu}  \leq n_{\lambda\mu},
		\]
		and $d_{\lambda\mu}  = 0$ if and only if $n_{\lambda\mu}=0$.
	\end{theorem}
	
	We restate Theorem~\ref{th:JS} using James's abacus in the setting when the column labeled by $\mu$ of the decomposition matrix is \textit{multiplicity-free} (that is, contains only zeros and ones). In the statement, we use the following notation.
	
	\begin{definition}\label{def:funkyA}
		Let $\mu$ be a $p$-regular partition and $\mathcal{B} = (B, b-ap, b)$ a hook triple. Define $J_{\mu}(\mathcal{B})$ to be the set of hook triples $(C,c-ap,c)$ such that
		\begin{enumerate}
			\item $C$ is a $\beta$-set of some partition $\nu$ with $d_{\nu\mu}>0$; and
			\item $b < c$ ; and
            \item the swap of $b-ap$ and $b$ in $B$ results in the same set as the swap of $c-ap$ and $c$ in $C$.
		\end{enumerate}
	\end{definition}

    \begin{remark}\label{re:hookInterpretation}
        Let $\lambda$ denote the partition of which $B$ is a $\beta$-set in Definition~\ref{def:funkyA} and suppose $g$ is the $ap$-hook of $\lambda$ corresponding to the pair $(b-ap, b)$. Then condition (3) can be rephrased as: the pair $(c-ap,c)$ corresponds to an $ap$-hook $h$ of $\nu$ such that  $[\lambda]\setminus g = [\nu]\setminus h$. Condition (2) then states that $\nu \rhd \lambda$.
    \end{remark}

    \begin{remark}\label{re:unique}
        Given $\mu$ and $\mathcal{B} = (B, b-ap, b)$ as in Definition~\ref{def:funkyA}, for each $\beta$-set $C$ of a partition, there is at most one choice of $c\in C$ such that $(C,c-ap,c)\in J_{\mu}(\mathcal{B})$. Furthermore, if $B$ is a $\beta$-set with respect to integer $r$, then there is $(C,c-ap,c)\in J_{\mu}(\mathcal{B})$ only if $C$ is also a $\beta$-set with respect to $r$, as discussed in Section~\ref{sec:abacus}.
    \end{remark}

It is routine work to translate Theorem~\ref{th:JS} to the following statement.
    
	\begin{corollary}\label{co:JS}
		Suppose that $\lambda\neq\mu$ are partitions of the same size with the latter being $p$-regular. Assume that the column of the decomposition matrix labeled by $\mu$ is multiplicity-free. Writing $B$ for a $\beta$-set of $\lambda$, we have
		\[
		n_{\lambda\mu} = \sum_{\mathcal{B} = (B,b-ap,b)} v_p(ap) \sum_{\mathcal{C}\in J_{\mu}(\mathcal{B})} (-1)^{\emp(\mathcal{B}) + \emp(\mathcal{C})+1},
		\]
        where the first sum runs over all integers $a\geq 1$ and $b$ such that $\mathcal{B} = (B,b-ap,b)$ is a hook triple.
		Moreover,
		\begin{align*}
			d_{\lambda\mu}= \begin{cases*}
				1 \textnormal{ if } n_{\lambda\mu}>0;\\
				0 \textnormal{ if } n_{\lambda\mu}=0.
			\end{cases*}
		\end{align*}
	\end{corollary}
	
	\begin{proof}
		The `moreover' part follows immediately from the last sentence of Theorem~\ref{th:JS} since the column of the decomposition matrix labeled by $\mu$ is multiplicity-free. Thus it remains to verify the formula for $n_{\lambda\mu}$. Using Theorem~\ref{th:JS} and the multiplicity-free assumption, we have
		\[
		n_{\lambda\mu} = \sum_{\substack{\nu \rhd \lambda \\ d_{\nu\mu} > 0}} \sum_{(g,h)\in \mathcal{H}(\lambda, \nu)} (-1)^{a(g) + a(h) + 1} v_p(|g|).
		\]
		Writing $\mathcal{H}(\cdot)$ for the set of hooks of a partition, we can rearrange this as
		\[
		n_{\lambda\mu} = \sum_{g\in \mathcal{H}(\lambda)} v_p(|g|) \sum_{\substack{\nu \rhd \lambda \\ d_{\nu\mu} > 0}} \quad \sum_{\substack{h\in \mathcal{H}(\nu) \\ [\lambda]\setminus g = [\nu]\setminus h}} (-1)^{a(g) + a(h) + 1} .
		\]
		After observing that we can restrict the first sum to $p$-divisible hooks $g$, the result follows from Definition~\ref{def:funkyA} (see also Remark~\ref{re:hookInterpretation}).
	\end{proof}

We can use this new formula to reduce Theorem~\ref{theorem:maintheorem} to a simpler statement: the columns of the decomposition matrices in Theorem~\ref{theorem:maintheorem} are multiplicity-free. This is proved using a downward induction (with respect to the dominance order) utilising the Jantzen--Schaper formula for multiplicity-free columns and a sign-reversing involution. We demonstrate the idea in the following example.

\begin{example}\label{ex:JSInvolution}
    Let $p=5$ and $\mu = (6,5^2,3,2,1)\vdash 22$. Assume that the column of the decomposition matrix labeled by $\mu$ is multiplicity-free. We show that for $\lambda=(6,4,2^5,1^2)$ we have $d_{\lambda\mu}=0$ assuming that for any $\nu\rhd\lambda$ of size $22$ we have $d_{\nu\mu}=1$ if and only if $\nu$ and $\mu$ have equal $p$-core and odd sequence. Let $B$ be the canonical $\beta$-set of $\lambda$. There are $3$ hook triples $\mathcal{B} = (B,b-ap,b)$, but, using Example~\ref{example:p=5gamma=3211allcolumns}, only for two of them $J_{\mu}(\mathcal{B})\neq \emptyset$, and for both of them $J_{\mu}(\mathcal{B})$ has size $2$; see Figure~\ref{fig:involution} for all the appearing hook triples. Moreover, in both cases, the elements $J_{\mu}(\mathcal{B})$ correspond to each other under the bijection $\widetilde{\Aone}$ (and its inverse $\widetilde{\Atwo}$) and their contributions to $n_{\lambda\mu}$ in Corollary~\ref{co:JS} cancel out; thus $n_{\lambda\mu}=d_{\lambda\mu}=0$.

    If instead $\lambda=(6,4,3,2^3,1^3)$ and $B$ is its canonical sets, then there are also two hook triples $\mathcal{B} = (B,b-ap,b)$ such that $J_{\mu}(\mathcal{B})\neq \emptyset$, but this time, in both cases, $J_{\mu}(\mathcal{B})=\left\{ \widetilde{\Aone}(\mathcal{B}) \right\}$; see Figure~\ref{fig:involution}. Moreover, both hook triples $\widetilde{\Aone}(\mathcal{B})$ contribute by $1$ to $n_{\lambda\mu}$; thus $n_{\lambda\mu}=2$ and $d_{\lambda\mu}=1$.
\end{example}

\begin{figure}[ht]
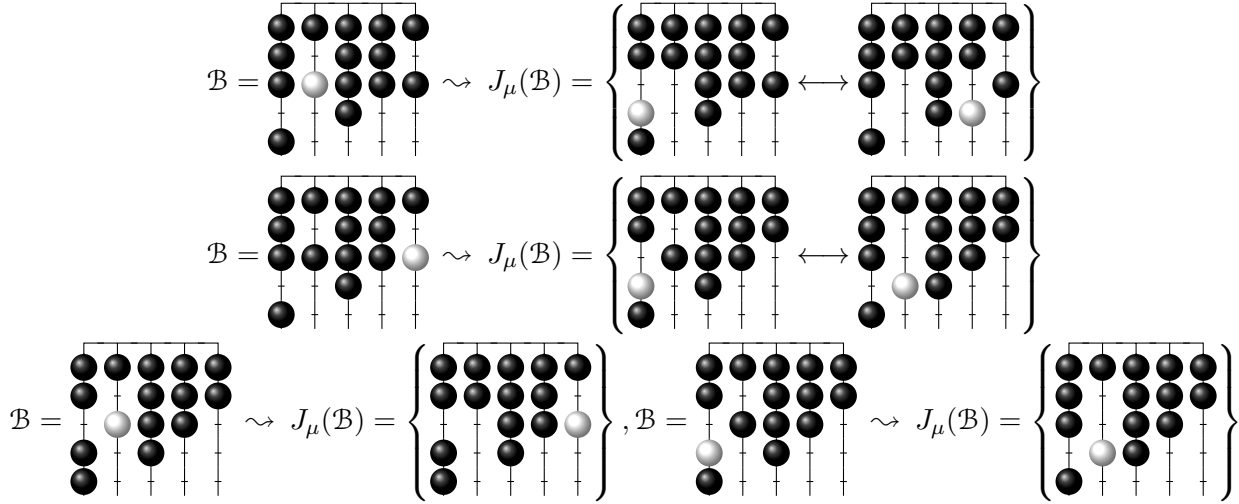

    \centering
    \[
     \mathcal{B} =\abacus(lmmmr,bbbbb,bnbbn,bobbb,nnbnn,bnnnn) \: \leadsto  \:J_{\mu}(\mathcal{B}) = \left\{ \abacus(lmmmr,bbbbb,bbbbn,nnbbb,onbnn,bnnnn) \longleftrightarrow \abacus(lmmmr,bbbbb,bbbbn,bnbnb,nnbon,bnnnn)\right\}
    \]
    \[
     \mathcal{B} =\abacus(lmmmr,bbbbb,bnbbn,bbbbo,nnbnn,bnnnn) \: \leadsto \: J_{\mu}(\mathcal{B}) = \left\{\abacus(lmmmr,bbbbb,bnbbb,nbbbn,onbnn,bnnnn) \longleftrightarrow \abacus(lmmmr,bbbbb,bnbbb,bnbbn,nobnn,bnnnn) \right\}
    \]
    \[
    \mathcal{B} =\abacus(lmmmr,bbbbb,bnbbb,nobbn,bnbnn,bnnnn) \: \leadsto \: J_{\mu}(\mathcal{B}) = \left\{\abacus(lmmmr,bbbbb,bbbbn,nnbbo,bnbnn,bnnnn) \right\}, 
    \mathcal{B} =\abacus(lmmmr,bbbbb,bnbbb,nbbbn,onbnn,bnnnn) \: \leadsto \: J_{\mu}(\mathcal{B}) = \left\{\abacus(lmmmr,bbbbb,bnbbb,bnbbn,nobnn,bnnnn)\right\}
    \]
    \caption{In the diagrams, we draw a hook triple $(B,b-ap,b)$ by drawing the abacus display of $B$ with bead at position $b$ highlighted, and take $b-ap$ to be immediately above $b$ (so in all drawn hook triples $a=1$). The first two lines display hook triples $\mathcal{B} = (B,b-ap,b)$ with $B$ the canonical $\beta$-set of $\lambda=(6,4,2^5,1^2)$ from Example~\ref{ex:JSInvolution} such that $J_{\mu}(\mathcal{B})\neq \emptyset$, and then the two-elements sets $J_{\mu}(\mathcal{B})$. There is an arrow between them to emphasise that they correspond to each other under the bijection $\widetilde{\Aone}$ (and its inverse $\widetilde{\Atwo}$). 
    On the last line, we replace $\lambda$ with $(6,4,3,2^3,1^3)$. The non-zero sets $J_{\mu}(\mathcal{B})$ then contain a single element given by $\widetilde{\Aone}(\mathcal{B})$.}
    \label{fig:involution}
\end{figure}

We are ready to reduce Theorem~\ref{theorem:maintheorem} to a simpler statement. Recall that for odd sequences $\theta$ containing a zero, the moreover part of Proposition~\ref{pr:equivalence max} proves that the maximal element in $\mathcal{E}_\theta(\gamma)$ is $p$-regular, so the following statement makes sense.
	
	\begin{proposition}\label{prop:decomposition}
		Let $\theta$ be a composition of length $p$ which contains a zero, and let $\gamma$ be a $p$-core partition. The column of the decomposition matrix labeled by the maximal partition of $\mathcal{E}_{\theta}(\gamma)$ has a one in precisely the rows labeled by partitions in $\mathcal{E}_{\theta}(\gamma)$ and zero elsewhere if and only if it is multiplicity-free.
	\end{proposition}
	
	\begin{proof}
        The `only if' direction is clear. To prove the converse, assume that the column of the decomposition matrix labeled by the maximal partition of $\mathcal{E}_{\theta}(\gamma)$, call it $\mu$, is multiplicity-free. We prove that for all $\lambda\vdash |\mu|$ we have $d_{\lambda\mu}=1$ if and only if $\lambda\in\mathcal{E}_{\theta}(\gamma)$ by downwards induction on $\lambda$ (with respect to the dominance order).

        Suppose that the statement is true for all partitions $\nu\rhd \lambda$. If $\lambda=\mu$, then the statement is true as $d_{\mu\mu}=1$ and $\mu\in\mathcal{E}_{\theta}(\gamma)$. Thus we can assume that $\lambda\neq \mu$ and apply Corollary~\ref{co:JS}.

        Let $B$ be the canonical $\beta$-set of $\lambda$. We can conclude the proof once we prove the following claim for any hook triple $\mathcal{B} = (B,b-ap,b)$.
        
        \begin{claim}\label{cl:js}
        The quantity
		\begin{equation*}
		    \tilde{n}_{\lambda\mu}(\mathcal{B}) := \sum_{\mathcal{C}\in J_{\mu}(\mathcal{B})} (-1)^{\emp(\mathcal{B}) + \emp(\mathcal{C})+1}
		\end{equation*}
		is $1$ if $\lambda\in \mathcal{E}_{\theta}(\gamma)$ and $\emp(\mathcal{B})$ is odd, \emph{and} is $0$ otherwise.
        \end{claim}
        
        Indeed, the claim would mean that $n_{\lambda\mu}=0$ if $\lambda\notin \mathcal{E}_{\theta}(\gamma)$, and is non-zero otherwise since in that case $\lambda$ is a non-maximal element of $\mathcal{E}_{\theta}(\gamma)$; thus, by Proposition~\ref{pr:equivalence max}, there is a hook triple $\mathcal{B} = (B,b-ap,b)$ with $\emp(\mathcal{B})$ odd, showing $n_{\lambda\mu}\geq v_p(ap)\tilde{n}_{\lambda\mu}(\mathcal{B})\geq 1$.

        Fix a hook triple $\mathcal{B} = (B,b-ap,b)$. By the inductive hypothesis, and the fact that $B$ is the canonical $\beta$-set of $\lambda$, we can replace (1) in Definition~\ref{def:funkyA} with
		
		\begin{enumerate}
			\item[(1')] $C$ is the canonical $\beta$-set of some partition $\nu\in \mathcal{E}_{\theta}(\gamma)$.
		\end{enumerate}
        Therefore, $J_{\mu}(\mathcal{B})\subseteq T^o_{\theta}(\gamma) \sqcup T^e_{\theta}(\gamma)$.
        
		To prove Claim~\ref{cl:js} we define an involution $\omega$ on $J_{\mu}(\mathcal{B})$ by
		\begin{align*}
			\omega(\mathcal{C}) = \begin{cases}
				\widetilde{\Aone}(\mathcal{C}) &\textnormal{if } \mathcal{C}\in T^o_{\theta}(\gamma);\\
				\widetilde{\Atwo}(\mathcal{C}) &\textnormal{if } \mathcal{C}\in T^e_{\theta}(\gamma) \textit{ and } \widetilde{\Atwo}(\mathcal{C})\neq \mathcal{B};\\
				\mathcal{C} &\textnormal{if } \mathcal{C}\in T^e_{\theta}(\gamma) \textit{ and } \widetilde{\Atwo}(\mathcal{C})= \mathcal{B}.
			\end{cases}
		\end{align*}
		
		We firstly check that $\omega(\mathcal{C})\in J_{\mu}(\mathcal{B})$ for any $\mathcal{C} = (C,c-ap,c)\in J_{\mu}(\mathcal{B})$. This is clear in the last case. In the first case, let $(C',c'-ap,c') = \widetilde{\Aone}(\mathcal{C})$. Then $C'$ is the canonical $\beta$-set of some partition $\nu\in \mathcal{E}_{\theta}(\gamma)$ by Lemma~\ref{le:A1runner} since $C$ is the canonical $\beta$-set of such a partition. Hence Definition~\ref{def:funkyA}(1') holds for $C'$.
        
        The definition of $\widetilde{\Aone}$ shows that $c'>c$; hence $c'>c>b$, and we see that Definition~\ref{def:funkyA}(2) holds for $\widetilde{\Aone}(\mathcal{C})$. It also shows that the swap of $c'-ap$ and $c'$ in $C'$ results in the same set as the swap of $c-ap$ and $c$ in $C$, which is, by assumption on $\mathcal{C}$, equal to the set obtained from $B$ by the swap of $b-ap$ and $b$; hence Definition~\ref{def:funkyA}(3) holds for $\widetilde{\Aone}(\mathcal{C})$. Therefore $\widetilde{\Aone}(\mathcal{C})\in J_{\mu}(\mathcal{B})$.

		The second case is similar. The difference is that we replace A1 by A2, use Lemma~\ref{le:A2runner} instead of Lemma~\ref{le:A1runner}, and change the argument for $c'>b$: the algorithm A2 applied to $\mathcal{C}$ performs the initial swap of $c-ap$ and $c$, which yields a set $B'$, and then one more (terminating) proper move. Since $B'$ can also be obtained from $B$ by swapping $b-ap$ and $b$, and $c>b$, this terminating proper move is either the swap of $b-ap$ and $b$ or happens before the algorithm swaps $b-ap$ and $b$; thus $c'\geq b$ with equality if and only if  $\widetilde{\Atwo}(\mathcal{C})=\mathcal{B}$.
		
		Using the inverse relation of A1 and A2 from Lemma~\ref{le:inverse}, we immediately see that $\omega$ is an involution. If $\mathcal{C}$ falls into the first two cases of the definition of $\omega$, from Corollary~\ref{cor:A1newPair}, we have $(-1)^{\emp(\omega(\mathcal{C}))} = - (-1)^{ \emp(\mathcal{C})}$. Thus, the contributions from $\mathcal{C}$ and $\omega(\mathcal{C})$ to $\tilde{n}_{\lambda\mu}(\mathcal{B})$ cancel out unless $\mathcal{C}$ belongs to the third case of the definition of $\omega$. By Lemma~\ref{le:inverse}, this happens if and only if $\mathcal{B}\in T_{\theta}^o(\gamma)$ and $\mathcal{C} = \widetilde{\Aone}(\mathcal{B})$. Note that in such a case, $\mathcal{C}\in J_{\mu}(\mathcal{B})$: Definition~\ref{def:funkyA}(1') holds by Lemma~\ref{le:A1runner} and the other two conditions are clear from the definition of A1.
        
        We deduce that if $\lambda\in \mathcal{E}_{\theta}(\gamma)$ and $\emp(\mathcal{B})$ is odd, then
		\[
		\tilde{n}_{\lambda\mu}(\mathcal{B}) =  (-1)^{\emp(\mathcal{B}) + \emp(\widetilde{\Aone}(\mathcal{B}))+1},
		\]
		and $\tilde{n}_{\lambda\mu}(\mathcal{B})=0$, otherwise. By Corollary~\ref{cor:A1newPair}(1), $\tilde{n}_{\lambda\mu}(\mathcal{B})=1$ in the former. This shows Claim~\ref{cl:js}, finishing the proof. 
	\end{proof}

\begin{remark}\label{re:projective remains}
    While Proposition~\ref{prop:decomposition} does not mention any twisted Foulkes modules, it is inevitably connected to them through Corollary~\ref{cor:FoulkesModules}. As we see in the proof of Theorem~\ref{theorem:maintheorem} and Corollary~\ref{cor:projective indecomposable} in Section~\ref{sec:proofs}, the just proved proposition shows that if the summand of the twisted Foulkes modules in Corollary~\ref{cor:FoulkesModules} is projective, then it is indecomposable, as stated in Section~\ref{sec:outlineproof}.
\end{remark}

\begin{remark}\label{re:q characteristic}
    By replacing $v_p$ by an appropriate function from \cite[Theorem~1.6]{FayersWeightThree08} throughout the proof, one shows that Proposition~\ref{prop:decomposition} holds for decomposition matrices for the Hecke algebra of odd quantum characteristic $p$ (not necessarily a prime) and arbitrary characteristic.
\end{remark}

\section{Brauer morphism}
\label{sec:Brauer}

The main goal of this section is to obtain the final piece of the proof of Theorem~\ref{theorem:maintheorem} in the form of the following theorem. As usual, $F$ is a field of odd prime characteristic $p$ throughout the section.

\begin{theorem}
\label{theorem:summandisprojective}
    Let $\gamma$ be a $p$-core partition, $\theta$ a composition of $k$ of length $p$ which contains a zero and $w=w_{\theta}(\gamma)$. Let $m = (|\gamma|+pw-k)/2$. The $FS_{2m+k}$-module $\left(\left(H^{(2^m)}_{C(\gamma) + \bm{w} - \theta} \boxtimes \sgn \right)\Ind^{S_{2m+k}}\right)_{C(\gamma) + \bm{w}}$ is projective.
\end{theorem}

We have seen in Corollary~\ref{cor:FoulkesModules} that $m$ is a non-negative integer, and so the statement makes sense. The same result also gives a formula for the ordinary character of this projective module (working over a discrete valuation ring of characteristic $0$); this formula and the fact that it is multiplicity-free are what allow us to eventually deduce Theorem~\ref{theorem:maintheorem} from Proposition~\ref{prop:decomposition} and Theorem~\ref{theorem:summandisprojective}. We show that the module in question is projective through the Brauer morphism. The proof can be seen as an abstraction of the proof of \cite[Proposition~5.1]{GiannelliWildonFoulkesandDecomposition15}. While the Brauer morphism is defined for any module of a group algebra, we only define it for $p$-permutation modules; in this setting, a definition more suitable for our purposes can be used.  

Let $G$ be a finite group. An $FG$-module $V$ is called a \textit{$p$-permutation module} if for every Sylow $p$-subgroup $Q\leq G$, there is a $Q$-invariant basis $\mathcal{I}_Q$ of $V$ (that is, a basis such that $Q\mathcal{I}_Q = \mathcal{I}_Q$). Equivalently, $V$ is a $p$-permutation module if $V\Res_Q$ is a permutation module. Note that since all Sylow $p$-subgroups of $G$ are conjugate in $G$, either of the conditions for $p$-permutation modules can be verified only for one Sylow $p$-subgroup of $G$. We mention the fact that $p$-permutation modules are closed under direct sum, summands, outer tensor products, induction and restriction (in fact, they are precisely summands of permutation modules \cite[(0.4)]{BroueP-perm85}, from which all of this follows). These definitions can also be made for $\mathcal{O}G$-modules where $\mathcal{O}$ is a discrete valuation ring, but we only need this generalization in the next section.

Let $V$ be a $p$-permutation module of $FG$. For a $p$-subgroup $R\leq G$, we use the notation $\mathcal{I}_R(V)$ for a $Q$-invariant basis of $V$ for some Sylow $p$-subgroup $Q\leq G$ containing $R$. The $F$-vector subspace $V(R)$ of $V$ spanned by the $R$-invariant elements of $\mathcal{I}_R(V)$ (that is, elements $v\in\mathcal{I}_R(V)$ such that $xv = v$ for all $x\in R$) inherits a structure of an $FN_G(R)$-module. Moreover, $V(R)$ is independent of the choice of $\mathcal{I}_R(V)$. The map $V\mapsto V(R)$ is called the \textit{Brauer morphism}.

\begin{example}\label{ex:sgn}
    If $V=\mathbf{1}$, the trivial module of $FG$, then, obviously, $\mathbf{1}(R)= \mathbf{1}\Res_{N_G(R)}$. Now suppose that $G=S_n$ (and recall that $p$ is odd). The one-dimensional sign module $\sgn$ of $FG$ is a $p$-permutation module: this is because all Sylow $p$-subgroups of $G$ lie inside $A_n$ and thus act on $\sgn$ as the identity. From this observation, for any $p$-subgroup $R$ of $G$, $\mathcal{I}_R(V)$ consists of a chosen non-zero vector of $\sgn$, and this vector is $R$-invariant; therefore, as for the trivial module, we get $\sgn(R) = \sgn\Res_{N_G(R)}$.   
\end{example}

The following relation of the Brauer morphism and vertices (defined in Section~\ref{se:rep}) proved in \cite[Theorem~3.2(1)]{BroueP-perm85} is a powerful tool we use to prove that the module in Theorem~\ref{theorem:summandisprojective} is projective.

\begin{theorem}\label{th:Brauer}
    Let $V$ be an indecomposable $p$-permutation $FG$-module and $R\leq G$ be a $p$-group. Then $V(R)\neq 0$ if and only if $R$ is contained in a vertex of $V$.
\end{theorem}

We summarise basic properties of the Brauer morphism below. For the reader's convenience we include short proofs. In the statement of (3) below, we write $\pi_G$ and $\pi_L$ for the projections from $G\times L$ onto $G$ and $L$, respectively. We further utilise the notation $A\otimes B$ for the set $\{a\otimes b \; | \; a\in A, \; b\in B \}$ and, for $x\in X\geq G$, also $\prescript{x}{}{V}$ for the $F(\prescript{x}{}{G})$-module with underlying $F$-vector space $\{ x\} \otimes V$ with $\prescript{x}{}{G}$-action given by $\prescript{x}{}{g} (x\otimes v) = x\otimes (gv)$.

\begin{lemma}\label{le:Brauer properties}
    Let $R$ and $T$ be $p$-subgroups of finite groups $G$ and $G\times L$, respectively. Suppose that $U,V$ are $p$-permutation $FG$-modules and $W$ is a $p$-permutation $FL$-module and $x\in X\geq G$. Then:
    \begin{enumerate}
        \item $(U\oplus V)(R) = U(R)\oplus V(R)$,
        \item $\left(\prescript{x}{}{V}\right)\left(\prescript{x}{}{R}\right) = \prescript{x}{}{(V(R))}$,
        \item $(V\boxtimes W)(T) = \left(V(\pi_G(T))\boxtimes W(\pi_L(T))\right)\Res_{N_{G\times L}(T)}$,
        \item $\left(V\Res_{N_G(R)}\right)(R) = V(R)$.
    \end{enumerate}
\end{lemma}

\begin{proof}
    Let $Q\geq R$ be a Sylow $p$-subgroup of $G$. We select bases $\mathcal{I}_Q(U) = \mathcal{I}_R(U)$, $\mathcal{I}_Q(V) = \mathcal{I}_R(V)$, $\mathcal{I}_{\pi_G(T)}(V)$ and $\mathcal{I}_{\pi_L(T)}(W)$ of $FG$-modules $U$, $V$ and $V$ and $FL$-module $W$, respectively. Then we can choose $\mathcal{I}_R(U\oplus V) = \mathcal{I}_R(U)\sqcup \mathcal{I}_R(V)$, $\mathcal{I}_{\prescript{x}{}{R}}(\prescript{x}{}{V})=\{x\}\otimes \mathcal{I}_R(V)$ and $\mathcal{I}_T(V\boxtimes W) = \mathcal{I}_{\pi_G(T)}(V)\otimes \mathcal{I}_{\pi_L(T)}(W)$. We obtain (1)--(3) by taking $R$-, $\prescript{x}{}{R}$- and $T$-invariant elements of these three sets, respectively.

    To prove (4), let $Q'\geq R$ be a Sylow $p$-subgroup of $N_G(R)$ and $Q\geq Q'$ be a Sylow $p$-subgroup of $G$. Then a $Q$-invariant basis of $V$ is also a $Q'$-invariant basis of $V$, and the result follows by taking $R$-invariant elements of such a basis.
\end{proof}

We introduce some more notation. If $K,L\leq G$ are finite groups, we denote by $G/_LK$ a set of left coset representatives of $G/K$ such that for any $x,y\in G/_LK$ with $LxK=LyK$, there exists $l\in L$ such that $x=ly$. Note that such a set always exists: for each orbit of the action of $L$ on the left cosets of $K$ in $G$, we select $x\in G$ from one of these cosets and add to $G/_LK$ elements $l_1x, l_2x,\dots, l_tx$, one for each coset in the given orbit (for some $l_1,l_2,\dots,l_t\in L$).

\begin{lemma}\label{le:induced p-basis}
    Let $K\leq G$ be two finite groups and $Q$ be a Sylow $p$-subgroup of $G$. If $V$ is a $p$-permutation module of $FK$, then $V\Ind^G = FG\otimes_{FK} V$ is a $p$-permutation module of $FG$ with a $Q$-invariant basis given by $\mathcal{I}_Q\left(V\Ind^G\right):=\{ g\otimes v \;\vert\; g\in G/_QK, \; v\in \mathcal{I}_{K\cap Q^g}(V)\}$. 
\end{lemma}

\begin{proof}
    The set $\mathcal{I}_Q\left(V\Ind^G\right)$ is a basis of $V\Ind^G = FG\otimes_{FK} V$, so we only need to check that it is invariant under the action of $Q$. For $q\in Q$ and $g\in G/_QK$ we have $qg =g'k$ for some $g'\in G/_QK$ and $k\in K$. Using the defining property of $G/_QK$, there is $q'\in Q$ such that $g' = q'g$, so $k=(g')^{-1}qg = g^{-1}(q')^{-1}qg\in K\cap Q^g$. Therefore, for any $v\in \mathcal{I}_{K\cap Q^g}(V)$, we have $q (g\otimes v) = g' \otimes (kv)\in \mathcal{I}_Q\left(V\Ind^G\right)$ since $\mathcal{I}_{K\cap Q^g}(V)$ is $k$-invariant.
\end{proof}

\begin{remark}\label{re:Mackey}
    Instead of `guessing' this formula, one can derive it from Mackey's theorem applied to $V\Ind^G\Res_Q$, which allows for finding an invariant basis for each direct summand separately.
\end{remark}

Using the notation $A^G$ for fixed points of group $G$ acting on some set $A$, we can identify elements of $\mathcal{I}_Q\left(V\Ind^G\right)$ fixed by the permutation action of $R\leq Q$.

\begin{lemma}\label{le:induced fixed}
    Let $K\leq G$ be two finite groups, $Q$ be a Sylow $p$-subgroup of $G$, $R\leq Q$ and $V$ be a $p$-permutation $FK$-module. Letting $\mathcal{I}_Q\left(V\Ind^G\right):=\{ g\otimes v \;\vert\; g\in G/_QK, \; v\in \mathcal{I}_{K\cap Q^g}(V)\}$, we have $\left(\mathcal{I}_Q\left(V\Ind^G\right)\right)^R = \{ g\otimes v \;\vert\; g\in G/_QK \textnormal{ s.t. } R^g\leq K, \; v\in (\mathcal{I}_{K\cap Q^g}(V))^{R^g} \}$.
\end{lemma}

\begin{proof}
    An element $g\otimes v\in \mathcal{I}_Q\left(V\Ind^G\right)$ is fixed by $R$ if and only if $Rg\subseteq gK$ and $g^{-1}Rg$ fixes $v$. These are precisely the defining conditions for $\left(\mathcal{I}_Q\left(V\Ind^G\right)\right)^R$.
\end{proof}

The explicit description of the $R$-invariant basis of $V\Ind^G$ allows us to obtain the formula for $V\Ind^G(R)$. We start with the case when $R$ is a normal subgroup of $G$.

\begin{lemma}\label{le:Normal p-group}
    Let $K\leq G$ be finite groups and $R\unlhd G$ be a $p$-group. Then for a $p$-permutation $FK$-module $V$ we have
    \begin{align*}
        V\Ind^G(R) = 
        \begin{cases*}
            V(R)\Ind^G &\textnormal{if} $R\leq K$\textnormal{;}\\
            0 & \textnormal{otherwise.}
        \end{cases*}
    \end{align*}
\end{lemma}

\begin{proof}
    Let $Q$ be a Sylow $p$-subgroup of $G$ (which necessarily contains $R$ as $R\unlhd G$). Since $R$ is a normal subgroup of $G$, from Lemma~\ref{le:induced fixed}, a basis of $V\Ind^G(R)$ is $\left(\mathcal{I}_Q\left(V\Ind^G\right)\right)^R = \{ g\otimes v \;\vert\; g\in G/_QK \textnormal{ s.t. } R\leq K, \; v\in (\mathcal{I}_{K\cap Q^g}(V))^{R} \}$, which is empty if $R\not\leq K$. If $R\leq K$, then $(\mathcal{I}_{K\cap Q^g}(V))^{R}$ is a basis of $V(R)$, and so $V\Ind^G(R)$ is spanned by $G/_QK\otimes V(R)$; thus it equals $V(R)\Ind^G$.
\end{proof}

We can immediately deduce the general formula for any subgroup $R$ of $G$. This is `Mackey's theorem for the Brauer morphism' stated (in a slightly reworded version) in Section~\ref{sec:outlineproof}. 

\setcounter{section}{1}
\setcounter{theorem}{6}

\begin{proposition}
    Let $K\leq G$ be two finite groups and $R\leq G$ be a $p$-group. Then for a $p$-permutation $FK$-module $V$, we have
    \[
    V\Ind^G(R) \cong \bigoplus_{\substack{x\in N_G(R)\backslash G/K \\ R\leq \prescript{x}{}{K}}} \left(\prescript{x}{}{V}\right)(R)\Ind^{N_G(R)}.
    \]
\end{proposition}

\begin{proof}
    By Lemma~\ref{le:Brauer properties}(4), the left-hand side equals $V\Ind^G\Res_{N_G(R)}(R)$, which can be expanded using Mackey's theorem (and Lemma~\ref{le:Brauer properties}(1)) as
    \[
    \bigoplus_{x\in N_G(R)\backslash G/K} \left((\prescript{x}{}{V})\Res_{\prescript{x}{}{K}\cap N_G(R)}\Ind^{N_G(R)}\right)(R).
    \]
    Since $R\unlhd N_G(R)$, we can apply Lemma~\ref{le:Normal p-group} to each summand. In particular, we can only sum over $x$ such that $R\leq \prescript{x}{}{K}$, in which case $\prescript{x}{}{K}\cap N_G(R) = N_{\prescript{x}{}{K}}(R)$, and we obtain 
    \[
    \bigoplus_{\substack{x\in N_G(R)\backslash G/K \\ R\leq \prescript{x}{}{K}}} \left( (\prescript{x}{}{V})\Res_{N_{\prescript{x}{}{K}}(R)}(R)\right)\Ind^{N_G(R)} = \bigoplus_{\substack{x\in N_G(R)\backslash G/K \\ R\leq \prescript{x}{}{K}}} (\prescript{x}{}{V})(R)\Ind^{N_G(R)},
    \]
    using Lemma~\ref{le:Brauer properties}(4) once more, as required.
\end{proof}

\setcounter{section}{6}
\setcounter{theorem}{8}

Note that one may replace the condition $R\leq \prescript{x}{}{K}$ in the direct sum in Proposition~\ref{pr:Mackey for Brauer} by $R^x\leq K$, and thus think of the direct sum as running through all $G$-conjugates of $R$ lying in $K$, up to $K$-conjugation.

The following example has already been proved in more generality in \cite[Lemma~4.2 and Proposition~4.3]{GiannelliWildonFoulkesandDecomposition15}, but we include it here to showcase Proposition~\ref{pr:Mackey for Brauer} in a friendly (and relevant) setting. We also use the example to introduce some more necessary notation: we write $[a,b]$ for the set of integers $c$ with $a\leq c\leq b$, $[m]$ for $[1,m]$, $S_{\Omega}$ for the symmetric group of permutations of a set $\Omega$ and $S_2\wr S_{[a,b]}$ for the centraliser of $(2a-1\; 2a)(2a+1\; 2a+2)\dots (2b-1\; 2b)$ in $S_{[2a-1, 2b]}$. In default, $S_n = S_{[n]}$ and $S_2\wr S_m = S_2\wr S_{[m]}\leq S_{2m}$. We also set $g_r:= \prod_{i=0}^{r-1} (ip + 1\; ip + 2\; \dots\; ip + p)$ for integer $r\geq 0$ and $R_r$ to be the cyclic group of order $p$ (or $1$ if $r=0$) generated by $g_r$. 

\begin{example}\label{ex:Foulkes}
    Let $G=S_{2m}$, $K=S_2\wr S_m$ and $V$ be the trivial $FK$-module (so $V\Ind^G$ is the Foulkes module $H^{(2^m)}$). If $r\geq 0$ is an integer such that $rp\leq 2m$, then a $G$-conjugate $R_r^x$ of $R_r$ lies in $K$ if and only if the $r$ (pairwise disjoint) $p$-cycles of $g_r^x$ can be divided into $r/2$ pairs of the form $(a_1\; a_2\; \dots \; a_p)$ and $(a_1\pm 1\; a_2 \pm 1\; \dots \; a_p \pm 1)$, where `+' is chosen when $a_i$ is odd, and `-' is chosen when $a_i$ is even. In particular, no conjugate of $R_r$ lies in $K$ if $r$ is odd, so by Proposition~\ref{pr:Mackey for Brauer} $H^{(2^m)}(R_r) = 0$ for $r$ odd.

    If $r$ is even, we pick $y\in S_{rp}$ such that $g_r^y = \prod_{i=0}^{r/2 - 1}k_i$ where $k_i := (2pi +1\; 2pi + 3\; \dots \; 2pi + 2p-1)(2pi+2\; 2pi+4 \; \dots \; 2pi+ 2p)$. Then all the possible $G$-conjugates of $g_r$ in $K$ are $K$-conjugated to $g_r^y$ (and are thus $K$-conjugated to each other): indeed, if $g_r^x = \prod_{i=0}^{r/2 - 1}h_i$, where each $h_i$ is the product of two $p$-cycles of the above form $(a_1\; a_2\; \dots \; a_p)$ and $(a_1\pm 1\; a_2 \pm 1\; \dots \; a_p \pm 1)$, one can define $z\in K$ on $2pi+1, 2pi+2,\dots, 2pi+2p$ for $0\leq i\leq r/2-1$ by requiring $h_i^z = k_i$, and then map the fixed points of $g_r^y$ to the fixed points of $g_r^x$ in any way so that $z(2j-1) = z(2j)-1$ for any fixed point $2j-1$ of $g_r^y$; this guarantees $g_r^y = g_r^{xz}$. 
    
    Thus, we can take $y$ to be the only element of $N_G(R_r)\backslash G/ K$ such that $R \leq \prescript{y}{}{K}$. Proposition~\ref{pr:Mackey for Brauer} and Example~\ref{ex:sgn} give us

    \begin{equation}\label{eq:Foulkes weak}
        H^{(2^m)}(R_r) \cong \mathbf{1}(R_r)\Ind^{N_G(R_r)} = \mathbf{1}\Res_{N_{\prescript{y}{}{K}}(R_r)}\Ind^{N_G(R_r)}.
    \end{equation}
    
    Since $N_G(R_r) = N_{S_{rp}}(R_r)\times S_{[rp+1,2m]}$ and, in turn, $N_{\prescript{y}{}{K}}(R_r) = N_{S_{rp}\cap\prescript{y}{}{K}}(R_r) \times S_2\wr S_{[rp/2 + 1, m]}$, we can rewrite $H^{(2^m)}(R_r)$ as
    \[
    \left(\mathbf{1}\Ind^{N_{S_{rp}}(R_r)}\right) \boxtimes \left(\mathbf{1}\Ind^{S_{[rp+1,2m]}}\right),
    \]
    using the trivial modules of $FN_{S_{rp}\cap\prescript{y}{}{K}}(R_r) = FN_{\prescript{y}{}{(S_2\wr S_{rp/2})}}(R_r)$ and $F\left(S_2\wr S_{[rp/2 + 1, m]}\right)$, respectively. We can identify the first module in the tensor product with $H^{(2^{rp/2})}(R_r)$ using \eqref{eq:Foulkes weak} with $m=rp/2$. Therefore, for even $r\leq 2m/p$ we have
    \begin{equation}\label{eq:Foulkes strong}
        H^{(2^m)}(R_r) \cong H^{(2^{rp/2})}(R_r)\boxtimes H^{(2^{m-rp/2})},
    \end{equation}
    where the second element in the tensor product is considered as a module of $FS_{[rp+1,2m]}\cong FS_{2m-rp}$.
\end{example}

To enhance this example, we use the following fact, stated in \cite[Theorem~2.14(iii)]{GiannelliDecomposition15} for indecomposable modules.

\begin{theorem}\label{th:blocks and Brauer}
    Let $V$ be a $p$-permutation $FS_n$-module belonging to the block labeled by $p$-residue sequence $\tau$ and suppose that $R$ is a $p$-subgroup of $S_n$ with support $[rp]$. Then the $FN_{S_n}(R)\cong FN_{S_{rp}}(R) \times FS_{[rp+1,n]}$-module $V(R)$ is isomorphic to a direct sum of modules $U\boxtimes W$ with $W$ lying in the block labeled by $p$-residue sequence $\tau - \bm{r}$. In particular, $V(R)$ is zero if $\tau - \bm{r}$ is not a $p$-residue sequence of a partition of $n-rp$.
\end{theorem}

\begin{corollary}\label{co:FoulkesSummands}
    If $r,m\geq 0$ are integers such that $rp\leq 2m$ and $\tau$ is a composition of length $p$, then
    \begin{align*}
        H^{(2^m)}_{\tau}(R_r)\cong
        \begin{cases*}
            H^{(2^{rp/2})}(R_r)\boxtimes H^{(2^{m-rp/2})}_{\tau - \bm{r}} & \textnormal{if} r \textnormal{is even};\\
            0 & \textnormal{if} r \textnormal{is odd}.
        \end{cases*}
    \end{align*}
\end{corollary}

\begin{proof}
    The result follows from Theorem~\ref{th:blocks and Brauer} after decomposing $H^{(2^m)}$ and $H^{(2^{m-rp/2})}$ into block summands in \eqref{eq:Foulkes strong}.
\end{proof}

We now repeat the same game with the twisted Foulkes modules.

\begin{proposition}\label{pr:twiseted Foulkes}
    If $r,m,k\geq 0$ are integers such that $rp\leq 2m+k$ and $\tau$ is a composition of length $p$, then
    \[
    \left( H^{(2^m)}_{\tau}\boxtimes \sgn\right)\Ind^{S_{2m+k}}(R_r) \cong \bigoplus_{\substack{0\leq t\leq r/2 \\ tp\leq m \\ (r-2t)p\leq k}} \left(H^{(2^{tp})}\boxtimes\sgn\right)\Ind^{S_{rp}}(R_r) \boxtimes \left(H^{(2^{m-tp})}_{\tau - 2\bm{t}} \boxtimes\sgn\right)\Ind^{S_{2m+k-rp}}, 
    \]
    where the group $S_{2m+k-rp}$ stands for $S_{[rp+1,2m+k]}$, the three $\sgn$ are modules of $FS_{[2m+1,2m+k]}$, $FS_{[tp+1,rp]}$ and $FS_{[2m+(r-2t)p+1,2m+k]}$, respectively, and $H^{(2^{m-tp})}_{\tau - 2\bm{t}}$ is a $FS_{[rp+1,2m+(r-2t)p]}$-module.
\end{proposition}

\begin{proof}
    Let $G=S_{2m+k}$ and $K=S_{2m}\times S_{[2m+1,k]}$. The $G$-conjugates $R_r^x$ of $R_r$ lying inside $K$ are, up to $K$-conjugation, characterised by the size of the intersection of $[2m]$ and the support of $R_r^x$. This size is a multiple of $p$, say $sp$ with $0\leq s\leq r$, $sp\leq 2m$ and $(r-s)p\leq k$. For such $s$, we choose $x_{s}\in S_{[sp+1,rp]\cup [2m+1,2m+(r-s)p]}\leq G$ which maps $[sp+1,rp]$ to $[2m+1,2m+(r-s)p]$, so that $R_r^{x_{s}}$ is generated by the product of $g_{s}$ and $(r-s)$-many disjoint $p$-cycles with support in $[2m+1,2m+(r-s)p]$. We write $K_s$ for $\prescript{x_s}{}{K}$. An application of Proposition~\ref{pr:Mackey for Brauer} gives
    \[
    \left( H^{(2^m)}_{\tau}\boxtimes \sgn\right)\Ind^{S_{2m+k}}(R_r) \cong \bigoplus_{\substack{0\leq s\leq r\\ sp\leq 2m \\ (r-s)p\leq k}}  \left( \prescript{x_s}{}{\left(H^{(2^m)}_{\tau} \boxtimes \sgn\right)} \right)\left(R_r\right)\Ind^{N_G\left(R_r\right)}. 
    \]
    Lemma~\ref{le:Brauer properties}(2) and (3) applied with $R=R_r^{x_s}$ and the fact that $\sgn$ is preserved by the Brauer morphism, as seen in Example~\ref{ex:sgn}, allow us to rewrite the summand labeled by $s$ as
    \[
    \left( \prescript{x_s}{}{\left(\left(H^{(2^m)}_{\tau}(R_{s}) \boxtimes \sgn\right)\Res_{N_K(R_r^{x_s})}\right)} \right)\Ind^{N_G\left(R_r\right)} = \left( \prescript{x_s}{}{\left(H^{(2^m)}_{\tau}(R_{s}) \boxtimes \sgn\right)} \right)\Res_{N_{K_s}\left(R_r\right)}\Ind^{N_G\left(R_r\right)}.
    \]
    Therefore,
    \begin{equation}\label{eq:Big one}
        \left( H^{(2^m)}_{\tau}\boxtimes \sgn\right)\Ind^{S_{2m+k}}(R_r) \cong \bigoplus_{\substack{0\leq s\leq r\\ sp\leq 2m \\ (r-s)p\leq k}}  \left( \prescript{x_s}{}{\left(H^{(2^m)}_{\tau}(R_{s}) \boxtimes \sgn\right)} \right)\Res_{N_{K_s}\left(R_r\right)}\Ind^{N_G\left(R_r\right)}. 
    \end{equation}
    By Corollary~\ref{co:FoulkesSummands}, we get non-zero summands only if $s$ is even, say $s=2t$, in which case, this summand can be written as
    \[
    \left(\prescript{x_{2t}}{}{\left(H^{(2^{tp})}(R_{2t})\boxtimes H^{(2^{m-tp})}_{\tau - 2\bm{t}} \boxtimes \sgn\right)}\right) \Res_{N_{K_{2t}}(R_r)}\Ind^{N_G(R_r)}.
    \]
    Recall that $x_{2t}\in S_{[2tp+1,rp]\cup [2m+1,2m+(r-2t)p]}\leq G$ maps $[2tp+1,rp]$ to $[2m+1,2m+(r-2t)p]$ (and vice versa). Writing $G' = S_{rp}$ and $K' = S_{2tp}\times S_{[2tp+1,rp]}$, we obtain factorisations $N_G(R_r) = N_{G'}(R_r)\times S_{[rp+1,2m+k]}$ and $N_{K_{2t}}(R_r) = N_{K'}(R_r)\times S_{[rp+1,2m+(r-2t)p]}\times S_{[2m+(r-2t)p+1,2m+k]}$. Therefore, we can split the summand as
    \begin{equation}\label{eq:summand}
    \left(H^{(2^{tp})}(R_{2t})\boxtimes \sgn\right)\Res_{N_{K'}(R_r)}\Ind^{N_{G'}(R_r)}\boxtimes \left(H^{(2^{m-tp})}_{\tau - 2\bm{t}} \boxtimes \sgn\right)\Ind^{S_{[rp+1,2m+k]}}.
    \end{equation}
    To conclude our proof, note that, if for fixed integers $r\geq 2t\geq 0$ we choose $m=tp$ and $k=(r-2t)p$ in \eqref{eq:Big one}, then the only valid choice of $s$ in the direct sum is $s=2t$, in which case $x_s$ can be chosen to be the identity. Thus \eqref{eq:Big one} simplifies to
    \[
    \left( H^{(2^{tp})}_{\tau}\boxtimes \sgn\right)\Ind^{S_{rp}}(R_r) \cong  \left(H^{(2^{tp})}_{\tau}(R_{2t}) \boxtimes \sgn\right) \Res_{N_{K'}(R_r)}\Ind^{N_{G'}(R_r)}.
    \]
    After summing over all compositions $\tau$ of length $p$, we obtain 
    \[
    \left( H^{(2^{tp})}\boxtimes \sgn\right)\Ind^{S_{rp}}(R_r) \cong  \left(H^{(2^{tp})}(R_{2t}) \boxtimes \sgn\right) \Res_{N_{K'}(R_r)}\Ind^{N_{G'}(R_r)}.
    \]
    This allows us to simplify the summand in \eqref{eq:summand} to
    \[
    \left( H^{(2^{tp})}\boxtimes \sgn\right)\Ind^{S_{rp}}(R_r)\boxtimes \left(H^{(2^{m-tp})}_{\tau - 2\bm{t}} \boxtimes \sgn\right)\Ind^{S_{[rp+1,2m+k]}},
    \]
    which completes the proof.
\end{proof}

We immediately deduce the following.

\begin{corollary}\label{co:twiseted Foulkes}
    If $r,m,k\geq 0$ are integers such that $rp\leq 2m+k$ and $\tau$ and $\rho$ are compositions of length $p$, then
    $\left(\left( H^{(2^m)}_{\tau}\boxtimes \sgn\right)\Ind^{S_{2m+k}}\right)_{\rho}(R_r)$ is isomorphic to
    \[
    \bigoplus_{\substack{0\leq t\leq r/2 \\ tp\leq m \\ (r-2t)p\leq k}} \left(H^{(2^{tp})}\boxtimes\sgn\right)\Ind^{S_{rp}}(R_r) \boxtimes \left(\left(H^{(2^{m-tp})}_{\tau - 2\bm{t}} \boxtimes\sgn\right)\Ind^{S_{2m+k-rp}}\right)_{\rho-\bm{r}}. 
    \]
\end{corollary}

\begin{proof}
    The result follows from Theorem~\ref{th:blocks and Brauer} after decomposing the modules $\left( H^{(2^m)}_{\tau}\boxtimes \sgn\right)\Ind^{S_{2m+k}}$ and $\left(H^{(2^{m-tp})}_{\tau - 2\bm{t}} \boxtimes\sgn\right)\Ind^{S_{2m+k-rp}}$ into block summands in Proposition~\ref{pr:twiseted Foulkes}.
\end{proof}

We are now ready to prove Theorem~\ref{theorem:summandisprojective} and conclude this section. We will assume in the proof that $F$ is the finite field of order $p$, that is, the residue field of the discrete valuation ring $\mathcal{O}$ of $p$-adic integers; the result for an arbitrary field $F$ of characteristic $p$ is obtained by field extension, which preserves projective modules.

\begin{proof}
    Suppose that the statement fails to hold for a given $p$-core $\gamma$ and a composition $\theta$ of $k$ of length $p$ which contains a zero. Recall that $w=w_{\theta}(\gamma)$ and $m=(|\gamma| + pw - k)/2$ are non-negative integers. Then $\left(\left(H^{(2^m)}_{C(\gamma) + \bm{w} - \theta} \boxtimes \sgn \right)\Ind^{S_{2m+k}}\right)_{C(\gamma) + \bm{w}}$ has a non-projective indecomposable summand $V$. So, $V$ has a non-trivial $p$-subgroup of $S_{2m+k}$ as a vertex. By Cauchy's theorem, this $p$-subgroup contains a cyclic subgroup of order $p$, say, after conjugating in $S_{2m+k}$ if necessary, that it contains $R_r$ for some $1\leq r\leq (2m+k)/p$. By Theorem~\ref{th:Brauer}, we have $V(R_r)\neq 0$ and thus
    \[
    \left(\left(H^{(2^m)}_{C(\gamma) + \bm{w} - \theta} \boxtimes \sgn \right)\Ind^{S_{2m+k}}\right)_{C(\gamma) + \bm{w}}(R_r)\neq 0.
    \]
    By Corollary~\ref{co:twiseted Foulkes}, there is an integer $0\leq t\leq r/2$ with $tp\leq m$ and $(r-2t)p\leq k$ such that
    \[
    \left(\left(H^{(2^{m-tp})}_{C(\gamma) + \bm{w} - \theta - 2\bm{t}} \boxtimes\sgn\right)\Ind^{S_{2m+k-rp}}\right)_{C(\gamma) + \bm{w}-\bm{r}}\neq 0.
    \]
    Observe that, $w\geq r$, as otherwise $C(\gamma) + \bm{w}-\bm{r}$ would not be a $p$-content of a partition; indeed, such a partition $\lambda$ would have to have $p$-core $\delta$ of size less than $\gamma$, but then $C(\gamma) + \bm{w}$ would be a $p$-content of partitions with $p$-core $\gamma$ and $p$-weight $w$ and of partitions with $p$-core $\delta$ and $p$-weight equal to he $p$-weight of $\lambda$ increased by $r$, contradicting Theorem~\ref{thm: p-residues}.
    
    Lifting this non-zero $FS_{2m+k-rp}$-module to an $\mathcal{O}S_{2m+k-rp}$-module we can apply Lemma~\ref{le:FoulkesModules} with $p$-core $\gamma$, integer sequence $\theta + 2\bm{t} - \bm{r}$ with sum of entries $k - (r-2t)p\geq 0$ and a non-negative integer $w-r$ (it applies as $m-tp = (|\gamma| + p(w-r) - k + (r-2t)p)/2$ is a non-negative integer) to compute its character. It is the sum of $\chi^{\lambda}$ labeled by partitions $\lambda$ with $p$-core $\gamma$, $p$-weight $w-r$ and odd sequence $\theta + 2\bm{t} - \bm{r}$; thus, if this module is non-zero, such a partition $\lambda$ exists. Since $\theta$ contains a zero, we need $t\geq r/2$, which, together with the earlier requirements for $t$, forces $t=r/2$. But then $\lambda$ has $p$-core $\gamma$, odd sequence $\theta$ and $p$-weight strictly smaller than $w=w_{\theta}(\gamma)$, a contradiction to the definition of $w_{\theta}(\gamma)$. This completes the proof by contradiction.
\end{proof}

\section{Proof and consequences of Theorem~\ref{theorem:maintheorem}}
\label{sec:proofs}

Let $\mathcal{O}$ be a discrete valuation ring of characteristic $0$ with residue field $F$ of odd prime characteristic $p$. Recall that (working modulo isomorphisms) reduction modulo the maximal ideal of $\mathcal{O}$ is a bijection between projective $\mathcal{O}S_n$-modules and projective $FS_n$-modules. In fact, it is also a bijection between $p$-permutation $\mathcal{O}S_n$-modules and $p$-permutation $FS_n$-modules (see \cite[Lemma~5.5.2]{BensonRepresentationCohomologyI91}). One obtains the former bijection by restricting the latter bijection to projective $\mathcal{O}S_n$-modules (which all sit among the $p$-permutation $\mathcal{O}S_n$-modules). With this in mind we can use Proposition~\ref{prop:decomposition}, Theorem~\ref{theorem:summandisprojective} and \eqref{eq:projective} to complete the proof of Theorem~\ref{theorem:maintheorem} which we restate here for the reader's convenience.

\setcounter{section}{1}
\setcounter{theorem}{4}

\begin{theorem}
  Let $p$ be an odd prime. Let $\theta=(\theta_0,\theta_1,\ldots, \theta_{p-1})$ be a composition such that at least one $\theta_i=0$, and let $\gamma$ be a $p$-core partition. There is a unique maximal element in $\mathcal{E}_\theta(\gamma)$ in the dominance order, which is $p$-regular. The corresponding column of the symmetric group decomposition matrix in characteristic $p$ has a one in precisely the rows labeled by partitions in $\mathcal{E}_\theta(\gamma)$ and zeros elsewhere.
\end{theorem}

\begin{proof}
    We have seen in Proposition~\ref{pr:equivalence max} (and preceding comments) that the first assertion is true, and in Proposition~\ref{prop:decomposition} that to prove the second assertion it suffices to show that the column of the decomposition matrix labeled by the maximal element of $\mathcal{E}_\theta(\gamma)$ is multiplicity-free. Let $\mu$ denote the maximal element of $\mathcal{E}_\theta(\gamma)$. We consider the $p$-permutation $\mathcal{O}S_{2m+k}$-module $V=\left(\left(H^{(2^m)}_{C(\gamma) + \bm{w} - \theta} \boxtimes \sgn \right)\Ind^{S_{2m+k}}\right)_{C(\gamma) + \bm{w}}$, where $k=|\theta|$, $w=w_{\theta}(\gamma)$ and $m=(|\gamma| + pw - k)/2$.
    
    Corollary~\ref{cor:FoulkesModules} shows that its ordinary character is $\sum_{\lambda\in\mathcal{E}_{\theta}(\gamma)} \chi^{\lambda}$. By Theorem~\ref{theorem:summandisprojective}, the reduction of the $p$-permutation $\mathcal{O}S_{2m+k}$-module $V$ is projective, and using the bijections discussed at the start of this section, $V$ is also projective. Thus its ordinary character is $\sum_{\nu\in M} \psi^{\nu}$, where the sum runs over some multiset $M$ of $p$-regular partitions of $2m+k$. Therefore, using \eqref{eq:projective}, we have
    \[
    \sum_{\lambda\in\mathcal{E}_{\theta}(\gamma)} \chi^{\lambda} = \sum_{\nu\in M} \psi^{\nu} = \sum_{\nu\in M}\sum_{\lambda\vdash 2m+k} d_{\lambda\nu}\chi^{\lambda}.
    \]
    It thus suffices to show that $\mu$ lies in $M$, as there are no repetitions on the left-hand side, and so any column labeled by $\nu\in M$ is multiplicity-free.

    Since $\mu\in\mathcal{E}_{\theta}(\gamma)$, there is $\nu\in M$ such that $d_{\mu\nu} = 1$ and so $\nu\unrhd\mu$. As $d_{\nu\nu}=1$, we also have $\nu\in\mathcal{E}_{\theta}(\gamma)$, but this means that $\nu=\mu$ as $\mu$ is the maximal element of $\mathcal{E}_{\theta}(\gamma)$. This concludes the proof. 
\end{proof}

\setcounter{section}{7}
\setcounter{theorem}{0}

\begin{remark}\label{re:overkill}
    Without applying Jantzen--Schaper formula (in particular, Proposition~\ref{prop:decomposition}), we can conclude from Theorem~\ref{theorem:summandisprojective} on its own, that not only is the column labeled by the maximal partition of $\mathcal{E}_{\theta}(\gamma)$ multiplicity-free, but its non-zero entries appear in some rows labeled by $\mathcal{E}_{\theta}(\gamma)$. So, in a way, Proposition~\ref{prop:decomposition} is stronger than needed, but its proof does not differ from that of a weaker sufficient version.
\end{remark}

Before continuing with twisted Foulkes modules, we briefly recall that due to the properties of adjustment matrices (see \cite[Theorem 6.35]{MathasHecke99}), the column in Theorem~\ref{theorem:maintheorem} labeled by, say, $\mu$ is a sum of columns of the decomposition matrix of a Hecke algebra of prime quantum characteristic $p$ over a field of characteristic $0$, one of which is the column labeled by $\mu$. Therefore, the column labeled by $\mu$ remains multiplicity-free when passing to this Hecke algebra. Remark~\ref{re:q characteristic} immediately describes this column.

\begin{corollary}\label{cor:Hecke}
    Theorem~\ref{theorem:maintheorem} (and Theorem~\ref{theorem:even arms}) hold for decomposition matrices of Hecke algebras of prime quantum characteristic $p$ over a field of characteristic $0$.
\end{corollary}

Returning to twisted Foulkes modules, the last paragraph of the proof of Theorem~\ref{theorem:maintheorem} showed that $P^{\mu}$ is a direct summand of the projective $FS_n$-module $\left(\left(H^{(2^m)}_{C(\gamma) + \bm{w} - \theta} \boxtimes \sgn \right)\Ind^{S_{2m+k}}\right)_{C(\gamma) + \bm{w}}$. In fact, we can say more.

\begin{corollary}\label{cor:projective indecomposable}
    Let $\theta$ be a composition of length $p$ which contains a zero, and let $\gamma$ be a $p$-core partition. If $\mu$ denotes the maximal element of $\mathcal{E}_{\theta}(\gamma)$, $k=|\theta|$, $w=w_{\theta}(\gamma)$ and $m=(|\gamma| + pw - k)/2$ then over a field of characteristic $p$, we have
    \[
    P^{\mu} \cong  \left(\left(H^{(2^m)}_{C(\gamma) + \bm{w} - \theta} \boxtimes \sgn \right)\Ind^{S_{2m+k}}\right)_{C(\gamma) + \bm{w}}.
    \]
\end{corollary}

\begin{proof}
    Let $\mathcal{O}$ be the ring of $p$-adic integers and $F$ be its residue field of size $p$. The statement is true when we consider both sides as $\mathcal{O}S_{2m+k}$-modules, as the ordinary characters agree by \eqref{eq:projective}, Theorem~\ref{theorem:maintheorem}, and Corollary~\ref{cor:FoulkesModules}. Reduction modulo the maximal ideal of $\mathcal{O}$ then proves the result over $F$. Finally, field extension proves the result over any field of characteristic $p$.
\end{proof}

We have thus identified some indecomposable (projective) summands of the twisted Foulkes modules $H^{(2^m; k)}$. We now move towards our final goal: identifying \emph{all} indecomposable summands of the Foulkes modules $H^{(2^m)}$. As the first step, we show that Corollary~\ref{cor:projective indecomposable} already found all the projective ones via the next lemma, attributed to Littlewood \cite{LittlewoodModular51}, and stated using the version from \cite[Theorem~2.7]{RichardsDecomposition96} (with arm and leg lengths swapped).

\begin{theorem}\label{thm:hook decomposition}
    Let $\mu$ be a $p$-regular partition and $\nu$ be a partition such that $|\mu|-|\nu|$ is divisible by $p$. Then
    \[
    \sum_{(\lambda, h)} (-1)^{a(h)}d_{\lambda\mu} = 0,
    \]
    where the sum runs over all pairs of a partition $\lambda$ of the same size as $\mu$ and a rim hook $h$ of $\lambda$ such that its removal from $\lambda$ gives rise to $\nu$.
\end{theorem}

Instead of using the full version of Theorem~\ref{thm:hook decomposition}, we will only need the following consequence (stated in the language of abaci).

\begin{corollary}\label{cor:hook decomposition}
    Let $B$ be a $\beta$-set of a $p$-regular partition $\mu$ and $\mathcal{B}=(B,b-ap,b)$ be a hook triple. There is a hook triple $\mathcal{C}=(C,c-ap,c)$ such that
    \begin{enumerate}
        \item $C$ is a $\beta$-set of some partition $\lambda$ with $d_{\lambda\mu}>0$; and
        \item $c<b$; and
        \item the swap of $b-ap$ and $b$ in $B$ results in the same set as the swap of $c-ap$ and $c$ in $C$; and
        \item $\emp(\mathcal{B})$ and $\emp(\mathcal{C})$ have different parities.
    \end{enumerate}
\end{corollary}

\begin{proof}
    Without loss of generality, we assume that $B$ is the canonical $\beta$-set of $\mu$. Let $g$ be the rim $p$-divisible hook of $\mu$ corresponding to $(b-ap,b)$, and let $\nu$ be the partition obtained from $\mu$ by the removal of $g$. Since $|\mu| - |\nu| =ap$ is divisible by $p$, we can apply Theorem~\ref{thm:hook decomposition} with $\mu$ and $\nu$ to obtain a signed sum of decomposition numbers equal to $0$. The pair $(\mu, g)$ contributes by $(-1)^{a(g)} = (-1)^{\emp(\mathcal{B})}$ to the sum since $d_{\mu\mu}=1$, and thus there is a pair $(\lambda, h)$ which contributes by $(-1)^{\emp(\mathcal{B})+1}d_{\lambda\mu}\neq 0$.
    
    Writing $\mathcal{C} = (C,c-ap,c)$ for the hook triple corresponding to $h$, with $C$ the canonical $\beta$-set of $\lambda$, we immediately see that it satisfies all required properties but (2). To deduce (2), as $d_{\lambda\mu}>0$, we have $\lambda\unlhd\mu$ and we cannot have equality as then $B=C$ and, by (3), $\mathcal{B}=\mathcal{C}$, contradicting (4). So $\lambda\lhd\mu$, that it, $c<b$.
\end{proof}

Recall that an odd bead is a bead at position $b$ such that the number of vacant positions $f<b$ is odd. We refer to beads which are not odd as \textit{even beads}. So partition $\lambda$ is even if and only if all beads in a $\beta$-set of $\lambda$ are even (which is further equivalent to the odd sequence of $\lambda$ being equal to $\bm{0}$). 

\begin{lemma}\label{le:even mu}
    Let $\mu$ be an even $p$-regular partition. If all partitions $\lambda$ such that $d_{\lambda\mu}>0$ are also even partitions, then all arm lengths of the $p$-divisible hooks of $\mu$ are even.
\end{lemma}

\begin{proof}
    We show the contrapositive. An example of the hook triples and $\beta$-sets used in the argument is in Figure~\ref{fig:contrapositive}. Let $B$ be the canonical $\beta$-set of $\mu$ and suppose that there is a hook triple $\mathcal{B} = (B,b-ap,b)$ with $\emp(\mathcal{B})$ odd. Let $\mathcal{C} = (C,c-ap,c)$ be the hook triple given by Corollary~\ref{cor:hook decomposition}. We claim that $c$ is an odd bead of $C$; this finishes the proof as then $C$ is a $\beta$-set of partition $\lambda$ which is not even but $d_{\lambda\mu}>0$.
    
    Let $B'$ be the $\beta$ set obtained from $C$ by swapping $c-ap$ and $c$ (and from $B$ by swapping $b-ap$ and $b$). Since $\emp(\mathcal{B})$ is odd, $\emp(\mathcal{C})$ is even, and thus $c$ is an odd bead of $C$ if and only if $c-ap$ is an even bead of $B'$. Since $b>b-ap>c-ap$ and all beads of $B$ are even, $c-ap$ is an even bead of both $B$ and $B'$, and so $c$ is an odd bead of $C$.
\end{proof}

\begin{figure}[ht]
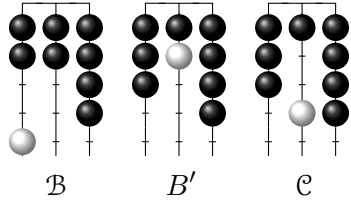

    \[
    \abacus(lmr,bbb,bbb,nnb,nnb,onn) \quad \abacus(lmr,bbb,bob,bnb,nnb,nnn) \quad \abacus(lmr,bbb,bnb,bnb,nob,nnn)
    \]
    \[
    \mathcal{B} \hspace{1.3cm} B' \hspace{1.3cm} \mathcal{C}
    \]
    \caption{The hook triple $\mathcal{B}$, $\beta$-set $B'$ and hook triple $\mathcal{C}$ from the proof of Lemma~\ref{le:even mu} with $\mu=(4^2,2)$. The hook triples are drawn as the underlying $\beta$-set displayed on the abacus, with the bead position from the triple highlighted, and with the empty position from the triple \emph{two} places above the bead. Using the notation from the proof, we have $\emp(\mathcal{B}) = 3$, $\emp(\mathcal{C})=2$ and $c-ap$ (highlighted in $B'$) is an even bead in both $B$ and $B'$ as there are no empty lower-numbered positions.}
    \label{fig:contrapositive}
\end{figure}

We are now ready to show that Corollary~\ref{cor:projective indecomposable} finds all projective indecomposable summands of the Foulkes module $H^{(2^m)}$. We will not go into the details, but the argument we use provides an alternative way of establishing Theorem~\ref{theorem:maintheorem} for $\theta = \bm{0}$ without using the Jantzen--Schaper formula. Indeed, omitting the second and last sentence of the upcoming proof, one shows that any projective indecomposable summand of $H^{(2^m)}$ must be of the form $P^{\mu}$ where $\mu$ is the maximal element of some set $\mathcal{E}_{\theta}(\gamma)$, and then it suffices to combine this observation with Theorem~\ref{theorem:summandisprojective} and Corollary~\ref{cor:FoulkesModules} to conclude Corollary~\ref{cor:projective indecomposable} and Theorem~\ref{theorem:maintheorem}.

\begin{proposition}\label{pr:projective summands}
    Let $m\geq 0$ be an integer. Let $V$ be an indecomposable summand of the Foulkes module $H^{(2^m)}$ in characteristic $p$. Then $V$ is projective if and only if $V = H^{(2^m)}_{C(\gamma) + \bm{w}}$ for some $p$-core partition $\gamma$ and $w=w_{\bm{0}}(\gamma)$ such that $2m = |\gamma| + pw$.
\end{proposition}

\begin{proof}
    As usual, we can work over the field $F$ with $p$ elements and take $\mathcal{O}$ to be the ring of $p$-adic integers. The `if' part is Corollary~\ref{cor:projective indecomposable} applied with $\theta = \bm{0}$. For the converse, suppose that $V\cong P^{\mu}$ for some $p$-regular partition $\mu$. Let $\gamma$ be its $p$-core and $w$ its $p$-weight (so $2m = |\gamma| + pw$). As discussed at the start of the section, we can lift $V$ to a projective $\mathcal{O}S_{2m}$-module, so that $V$ is a summand of the $\mathcal{O}S_{2m}$-module $H^{(2^m)}_{C(\gamma) + \bm{w}}$. Passing to ordinary characters, by Lemma~\ref{le:FoulkesModules} (applied with $\theta = \bm{0}$), we have
    \[
    \psi_{\mu} = \sum_{\lambda} \chi^{\lambda},
    \]
    where the sum is taken over \emph{some} even partitions $\lambda$ with $p$-core $\gamma$ and $p$-weight $w$. By \eqref{eq:projective}, $d_{\lambda\mu}=1$ for partitions $\lambda$ in the sum, and $d_{\lambda\mu}=0$ otherwise. In particular, $\mu$ is even, and Lemma~\ref{le:even mu} shows that all arm lengths of $p$-divisible hooks of $\mu$ are even. Using Proposition~\ref{pr:equivalence max}, $\mu$ is the maximal element of $\mathcal{E}_{\bm{0}}(\gamma)$, so $w=w_{\bm{0}}(w)$. Corollary~\ref{cor:projective indecomposable} then shows that $H^{(2^m)}_{C(\gamma) + \bm{w}}$ is indecomposable; thus it equals $V$, as required.
\end{proof}

To extend Proposition~\ref{pr:projective summands} to a description of all indecomposable summands of the Foulkes module, we use \cite[Theorem~1.2]{GiannelliWildonFoulkesandDecomposition15}. We state it here in a slightly modified version: the main part follows from letting $k=0$ in the original version (which forces $r=2t$), while the `moreover' part is implicit from the proof of \cite[Theorem~1.2]{GiannelliWildonFoulkesandDecomposition15}.

\begin{theorem}\label{th:Foulkes vertices}
    Let $m\geq 0$ be an integer. If $V$ is an indecomposable non-projective summand of $H^{(2^m)}$ over a field of characteristic $p$, then $V$ has a vertex given by a Sylow $p$-subgroup $Q$ of $S_2\wr S_{tp}$ for some $t\geq 1$ such that $tp\leq m$. Moreover, $V(R_{2t}) \cong H^{(2^{tp})}(R_{2t})\boxtimes W$ as modules of $FN_{S_{2m}}(R_{2t})\cong FN_{S_{2tp}}(R_{2t})\times FS_{[2tp+1,2m]}$, where $W$ is a projective indecomposable summand of $H^{(2^{m-tp})}$.
\end{theorem}

\begin{remark}\label{re:non-zero}
    The module $V(R_{2t})$ in the statement is non-zero, since $R_{2t}$ is contained in a vertex of $V$ (a suitable conjugate of $Q$) and so Theorem~\ref{th:Brauer} applies.
\end{remark}

The proof of our final main result, restated here for the reader's convenience, follows.

\setcounter{section}{1}
\setcounter{theorem}{5}

\begin{theorem}
    Let $p$ be an odd prime and $m\geq 0$ an integer. Let $\tau$ be the $p$-content of an even partition of $2m$. Then over a filed of characteristic $p$, $H^{(2^m)}_{\tau}$ is indecomposable. That is,
    \[
    H^{(2^m)} = \bigoplus_{\tau} H^{(2^m)}_{\tau},
    \]
    is the decomposition of the Foulkes module $H^{(2^m)}$ into indecomposable summands (where the sum runs over all distinct $p$-contents $\tau$ of even partitions of $2m$).
\end{theorem}

\begin{proof}
    Let $V$ be a indecomposable summand of $H^{(2^m)}_{\tau}$. If $V$ is projective, then $V = H^{(2^m)}_{\tau}$ by Proposition~\ref{pr:projective summands} and we are done. So suppose that all indecomposable summands of $H^{(2^m)}_{\tau}$ are non-projective. Let $t$ be the positive integer from Theorem~\ref{th:Foulkes vertices} applied to $V$. The `moreover' part of the same statement and Theorem~\ref{th:blocks and Brauer} shows that $V(R_{2t}) \cong H^{(2^{tp})}(R_{2t})\boxtimes W$, where $W$ is a projective indecomposable summand of $H^{(2^{m-tp})}_{\tau - 2\bm{t}}$. By Proposition~\ref{pr:projective summands}, we have $W= H^{(2^{m-tp})}_{\tau - 2\bm{t}}$ and $\tau - 2\bm{t}= C(\gamma) + \bm{w}$ for some $p$-core partition $\gamma$ and $w=w_{\bm{0}}(\gamma)$. Utilizing Theorem~\ref{thm: p-residues}, $\gamma$ is the $p$-core of partitions with $p$-content $\tau$ and so $t$ and $W$ depend only on $\tau$, not on $V$.

    That is, if $\gamma$ is the $p$-core of any partition with $p$-content $\tau$ and $t = (2m - |\gamma| - pw_{\bm{0}}(\gamma))/2p$, any indecomposable summand $V$ of $H^{(2^m)}_{\tau}$ has as a vertex a Sylow $p$-subgroup of $S_2\wr S_{tp}$ and $V(R_{2t}) \cong H^{(2^{tp})}(R_{2t})\boxtimes H^{(2^{m-tp})}_{\tau - 2\bm{t}}$ (which is non-zero; see Remark~\ref{re:non-zero}). By Corollary~\ref{co:FoulkesSummands}, we can rewrite this as $V(R_{2t}) \cong H^{(2^m)}_{\tau}(R_{2t})$. As the right-hand side is the direct sum of $U(R_{2t})$ as $U$ goes through all the indecomposable summands of $H^{(2^m)}_{\tau}$ (including $V$), and all $U(R_{2t})$ are non-zero, $H^{(2^m)}_{\tau}$ is either zero or indecomposable. As there is a partition of $2m$ with $p$-content $\tau$, it is the latter, finishing the proof.
\end{proof}

\setcounter{section}{7}
\setcounter{theorem}{9}

\begin{example}\label{ex:indecomposable}
    Let $p=5$ and $m=7$. There are $6$ different $p$-contents of even partitions of $14$:
    \begin{alignat*}{3}
\tau_1 &= (3,3,3,3,2) &\qquad
\tau_2 &= (4,3,2,2,3) &\qquad
\tau_3 &= (3,3,2,3,3) \\[6pt]
\tau_4 &= (3,2,2,3,4) &\qquad
\tau_5 &= (3,2,3,3,3) &\qquad
\tau_6 &= (2,3,4,3,2).
\end{alignat*}
Thus $H^{(2^7)}$ has $6$ indecomposable summands, namely,
\[
H^{(2^7)} = \bigoplus_{i=1}^6 H^{(2^7)}_{\tau_i}.
\]
\end{example}

There is a natural guess for a generalization of this decomposition to twisted Foulkes modules $H^{(2^m; k)}$. We state it here as a conjecture.

\begin{conjecture}\label{con:indecomposables}
    Let $m,k\geq 0$ be integers. Let $\tau$ and $\theta$ be the $p$-content and the odd sequence of a partition of $2m+k$, respectively, with $k=|\theta|$ . Then $\left(\left( H^{(2^m)}_{\tau-\theta} \boxtimes \sgn\right)\Ind^{S_{2m+k}}\right)_{\tau}$ is indecomposable. That is,
    \[
    H^{(2^m; k)} = \bigoplus_{(\tau, \theta)} \left(\left( H^{(2^m)}_{\tau-\theta} \boxtimes \sgn\right)\Ind^{S_{2m+k}}\right)_{\tau},
    \]
    is the decomposition of the twisted Foulkes module $H^{(2^m; k)}$ into indecomposable summands (where the sum runs over all pairs of $\tau$ and $\theta$ as above).
\end{conjecture}

Many of the ingredients of the proof of Theorem~\ref{thm:indecomposables} generalize to twisted Foulkes modules $H^{(2^m; k)}$, for instance, the result of Giannelli--Wildon stated here as Theorem~\ref{th:Foulkes vertices} and our Corollary~\ref{cor:projective indecomposable}. However, the combinatorics of partitions with non-zero odd sequences prevents us from straightforwardly applying Littlewood's formula in Theorem~\ref{thm:hook decomposition} as done for partitions with zero odd sequence in Lemma~\ref{le:even mu}. In turn, while we have found some projective indecomposable summands of $H^{(2^m; k)}$, we are unable to prove that we have found all of them. Note that a generalization of Lemma~\ref{le:even mu} to non-zero odd sequences would not only likely prove Conjecture~\ref{con:indecomposables}, but it could also provide an alternative proof of Theorem~\ref{theorem:maintheorem} without using the Jantzen--Schaper formula (see the paragraph before Proposition~\ref{pr:projective summands}).

\section*{Acknowledgment}
The authors would like to thank Mark Wildon for providing computational data at the early stage of the project and for a discussion about Theorem~\ref{thm:indecomposables}, Gunter Malle and Liron Speyer for their questions, which led to the enumeration of the newly found columns in a given block, Stacey Law and Lorenzo Putignano for their valuable suggestions on improving the manuscript and Kai Meng Tan for suggesting a further entry for the literature background. The second author is grateful to Bim Gustavsson and Stacey Law for many in-depth conversations while working on a connected project.

The second author was supported by the LMS Early Career Fellowship ECF-2025-26 at the University of Birmingham and is currently funded by the FY2025 JSPS Postdoctoral Fellowship for Research in Japan (Short-term(PE)) PE25723 at the Okinawa Institute of Science and Technology.

\bibliographystyle{alpha}
	\bibliography{MSNrefs26}

\end{document}